\documentclass[a4paper,12pt]{article}
\usepackage{amsmath}
\usepackage{amssymb}
\usepackage{tabularx}
\usepackage{enumerate}
\usepackage{graphicx}
\usepackage{indentfirst}
\usepackage{subfigure}
\usepackage{epsfig}
\usepackage{graphics}
\usepackage{cases}
\usepackage[compress]{cite}
\usepackage{txfonts}
\usepackage{geometry}
 \usepackage{epstopdf}
 \usepackage{lineno}
 \usepackage{color}
 \usepackage{float}
\usepackage{multirow}

\newtheorem{thm}{Theorem}[section]
\newtheorem{lem}[thm]{Lemma}

\newtheorem{exm}{Example}[section]

\newtheorem{defi}{Definition}[section]
\newtheorem{pppp}{Proof}
\newtheorem{case}{Case}[section]

\newcommand{\qed}{\hspace{1em}\mbox{\raisebox{0.65ex}{\fbox{}}}}

\numberwithin{equation}{section}

\newcommand{\be}{\begin{equation}}
\newcommand{\ee}{\end{equation}}
\newcommand\bes{\begin{eqnarray}} \newcommand\ees{\end{eqnarray}}
\newcommand{\bess}{\begin{eqnarray*}}
\newcommand{\eess}{\end{eqnarray*}}

\newcommand{\epf}{\mbox{}\hfill $\Box$}

\begin{document}

\thispagestyle{empty}

\title{Dengue fever model with impulsive intervention in a periodically varying environment\thanks{The work is partially supported by the NNSF of China (Grant No. 12271470).}}

\date{\empty}

\author{Han Zhang$^1$, Inkyung Ahn$^{2}$ and Zhigui Lin$^1 \thanks{Corresponding author. Email: zglin@yzu.edu.cn (Z. Lin).}$\\
{\small 1 School of Mathematical Science, Yangzhou University, Yangzhou 225002, China}\\
{\small 2 Department of Applied Mathematical Sciences, Korea University, Sejong 30019, Republic of Korea}}

 \maketitle
\begin{quote}
\noindent
{\bf Abstract.} { \footnotesize\small
Pulse interventions represent a highly effective measure for infection control, as they influence disease transmission through short-term actions. In addition, the habitat ranges of dengue vectors and hosts exhibit periodic variations driven by environmental and climatic factors. To investigate the effects of impulsive interventions and domain evolution on disease transmission, we propose a dengue fever reaction-diffusion model that incorporates impulsive perturbations in a periodically varying domain. By applying the Poincar$\acute{e}$ map and the Krein-Rutman theorem, we establish the existence of the principal eigenvalue for the periodic eigenvalue problem with impulses, thereby extending previous studies on reaction-diffusion equations in fixed domains without impulsive effects. Sufficient conditions governing the long-term dynamics of periodic solutions are derived using the comparison principle and monotone iteration theory. Numerical simulations corroborate the theoretical findings and elucidate the effects of impulsive-intervention intensity and periodic domain evolution on disease transmission patterns. Our results indicate that increasing the intensity of pulse interventions suppresses disease transmission, whereas a larger magnitude of domain variation impedes disease control.
}


\end{quote}

\section{Introduction and main results}
Infectious diseases have persisted as a grave threat to global public health. Despite steady medical advances, factors such as population migration, urbanization, and environmental pollution have amplified the spread of these diseases. Outbreaks such as SARS, H7N9 influenza, Cholera, and the recent COVID-19 pandemic have each inflicted heavy death tolls, economic damage, and societal disruption\cite{Sharif2024}.

As one of the many infectious diseases, dengue fever poses a serious threat. Dengue fever is a viral infection primarily transmitted by \textit{Aedes aegypti} and \textit{Aedes albopictus} mosquitoes. Severe dengue can lead to shock, respiratory distress, severe bleeding, and even death. Half of the global population is at risk of dengue virus infection\cite {Dengue2025}. In recent decades, the distribution of disease vectors has shifted due to factors such as El Ni$\tilde{n}$o and climate change, which leads indirectly to a significant increase in the global incidence of dengue fever. For example, the number of global dengue cases reported by the World Health Organization (WHO) increased from 505,430 in 2000 to 5.2 million in 2019\cite{Dengue2023-1,Dengue2023}. Based on global dengue surveillance data \cite{WHO2026}, the global numbers of reported total, confirmed, severe, and fatal cases from 2021 to 2025 are plotted, with the results shown in Fig.~1(a)-(b). The data reveal an overall worsening trend in the global dengue epidemic. Especially in 2024, total reported cases surged to 14 million, with fatalities nearing 10,000, which underscores the particular severity of the outbreak that year. Apart from monitoring and prevention, therefore, the formulation and application of mathematical models for predicting outbreak dynamics hold practical significance.
\begin{figure}[htbp]
\centering
\subfigure[\scriptsize Graph of total cases and confirmed cases
]{ {
\includegraphics[width=0.9\textwidth]{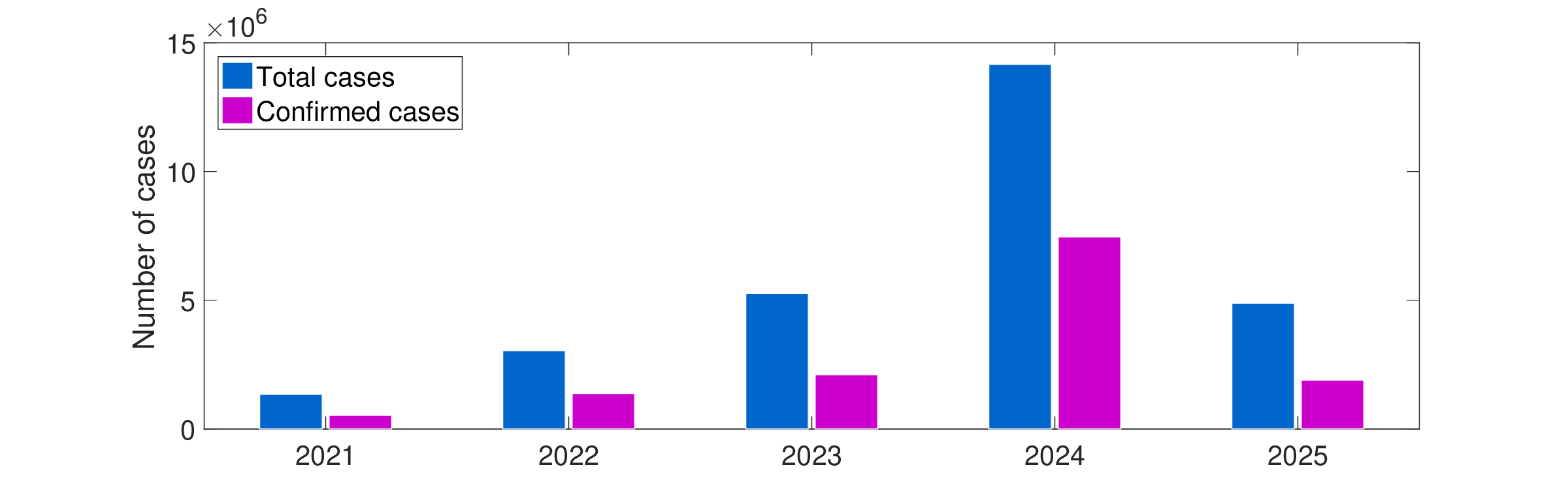}
} }
\subfigure[\scriptsize Graph of severe cases and total deaths]{ {
\includegraphics[width=0.9\textwidth]{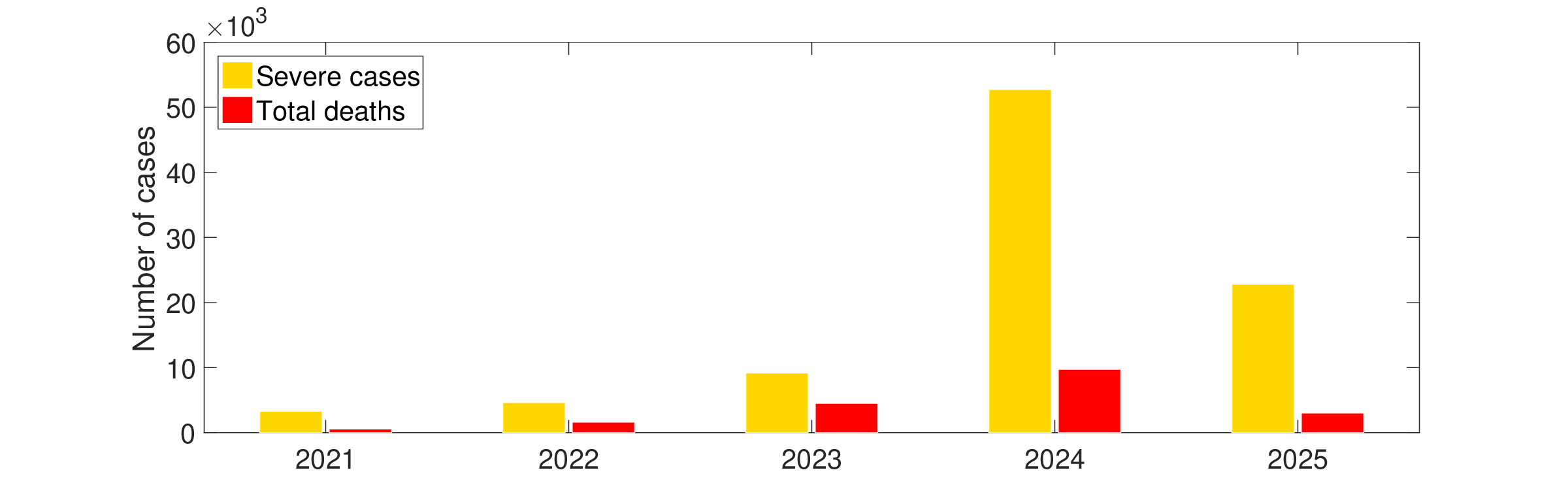}
} }
\renewcommand{\figurename}{\scriptsize Fig.}
\caption{\scriptsize Global reported dengue cases, 2021-2025. Data source: https://worldhealthorg.shinyapps.io/dengue\_global/.}
\end{figure}

Mathematical models have become an indispensable tool in public health research by revealing disease transmission mechanisms, identifying critical factors, formulating prevention and control strategies, and assessing epidemic dynamics\cite {Li2018}. To provide a crucial theoretical foundation for the effective management and ultimate eradication of dengue fever, the development and analysis of mathematical models are of substantial practical significance. In 1970, Fischer and Halstead first employed a mathematical model to investigate dengue fever, aiming to predict the risk of reinfection among different age groups under ``multi-serotype exposure''\cite{Fischer1970}. Over the decades that followed, many scholars have further developed mathematical models of dengue fever. In\cite{FengZhilan1997}, for instance, a dual-strain dengue model was formulated, with emphasis on the dynamical mechanisms governing competition-driven exclusion versus short-term coexistence of the virus strains. A non-autonomous dengue model was proposed in \cite{Coutinho2005} to simulate seasonal transmission dynamics, in which all dynamical features are governed by an approximate transmission-threshold condition. A dengue model with nonlocal time delays and spatial dispersal was developed in \cite{wangwendi2011}. By defining the basic reproduction number $R_0$ as a threshold parameter, it was shown that whether the disease dies out or persists depends on the magnitude of $R_0$. In addition, a classical SIR-SI dengue model was presented in \cite{Tewa2009}.

By simplifying the SIR-SI model, Zhu, Lin, and Zhang\cite{zhumin2018} considered the following dengue model
\begin{eqnarray}
\left\{ \begin{array}{l}
{I'_H} = \frac{{{\beta _H}b}}{{{N_H} + L}}(\frac{{{b_H}{N_H}}}{{{m_H}}} - {I_H}){I_V} - {m_H}{I_H} - {r_H}{I_H},\\
{I'_V} = \frac{{{\beta _V}b}}{{{N_H} + L}}(\frac{\Lambda }{{{m_V}}} - {I_V}){I_H} - {m_V}{I_V},\\
0 < {I_H}(0) < \frac{{{b_H}{N_H}}}{{{m_H}}},\;0 < {I_V}(0) < \frac{\Lambda }{{{m_V}}},
\end{array} \right.
\label{a03}
\end{eqnarray}
where variables and parameters for model (\ref{a03}) are defined in Table 1. They investigated the spatial dispersal and domain evolution.\vspace{-0.5\baselineskip}
\begin{table}[htbp]
\renewcommand{\arraystretch}{1.1}
  \centering
  \setlength{\belowcaptionskip}{10pt}
   {\caption{\; The definitions of all the variables or parameters in system (\ref{a03})}}
  \scalebox{1}{
 \begin{tabular}{l p{10cm}}
    \hline
    {\bf Variables or parameters} & {\bf Definitions} \\
    \hline
    \multirow{2}{*}{$I_H, I_V$}      & The densities of the infected human and mosquito populations at time $t$, respectively; \\
    $N_H$ & The total population size of humans; \\
    $\Lambda$      & The recruitment rate of mosquitoes; \\
    $b$     & The mean rate of mosquito bites per mosquito; \\
    $L$     & The density of other hosts (poultry and livestock); \\
    $m_H,m_V$     & The mortality rate of human and mosquito, respectively; \\    $b_H,r_H$     & The birth and recovery rate of human, respectively. \\     \multirow{2}{*}{ $\beta_H,\beta_V$}     & The contact transmission probabilities from infectious mosquitoes to susceptible humans, and from infectious humans to susceptible mosquitoes, respectively. \\
    \hline
  \end{tabular}
  }
  \label{tu1}
\end{table}

Since mosquito survival and reproduction are influenced by seasonal factors such as temperature and humidity, their distribution ranges exhibit periodic variation \cite {Zitko2014}. In\cite{liyuepeng2023}, for example, Li et al. protracted monthly distribution changes for three \textit{Aedes} mosquito species (\textit{Aedes aegypti}, \textit{Aedes albopictus}, and
\textit{Aedes vexans}) and observed that their distribution ranges display a periodic pattern. Such periodic variation in mosquito distribution is a major factor driving the periodic transmission patterns of dengue fever. Consequently, integrating domain evolution into models is now an important direction for understanding the spread of infections. For instance, a dengue model that incorporates regional periodic changes is developed in \cite{zm2019} to explore how such changes affect the transmission of dengue fever. In\cite{Wangjie2025}, Wang et al. analyzed a May-Nowak type degenerate system on a periodically varying domain and revealed the fundamental mechanisms of viral propagation. Beyond epidemic models, research has been conducted on population dynamics across various domains, as can be found in \cite{Adam2019,Fangjian2024,zhanghan2025}.

In addition to seasonal changes, impulsive interventions (e.g., targeted vector suppression, deployment of Wolbachia-carrying male mosquitoes, or emergency vaccination) also play a critical role in regulating dengue transmission \cite{Wijaya2021}. Thus, it is necessary to consider the impact of impulsive disturbances during model development. The impulsive reaction-diffusion model proposed by Lewis and Li in \cite{Lewis2012} is considered a groundbreaking contribution to impulsive theory in biomathematics and provided a foundational framework for subsequent developments in this direction. For an SIRS epidemic model, Nie et al.\cite{Nie2013} introduced a state-dependent impulsive control mechanism and proved the existence and orbital stability of a positive periodic solution. In \cite{Fazly2017}, Fazly and Lewis developed impulsive differential equations on a higher-dimensional domain $\Omega\in R^N(N\ge 0)$, which extends the results of \cite{Lewis2012}. Zhang et al.\cite{zhangYuron2024} developed an impulsive reaction-diffusion model with environmental variations and proved the existence of traveling wave solutions as well as the asymptotic spreading properties of the populations. By developing a nonlocal impulsive diffusion model, Lu and Wang\cite{Lu2026} conducted an in-depth study on the spatial distribution and propagation dynamics of populations characterized by synchronous reproduction and two-stage diffusion.

Motivated by the works above, we take into account a periodically evolving domain along with an impulsive intervention, and thus extend model (\ref{a03}) to the following system:
{\small \begin{eqnarray}
\left\{ {\begin{array}{*{20}{l}}
{\frac{{\partial {I_H}}}{{\partial t}} + {\bf{a}} \cdot \nabla {I_H} + {I_H}(\nabla  \cdot {\bf{a}}) = {d_H}\Delta {I_H} + \frac{{{\beta _H}b}}{{{N_H} + L}}(\frac{{{b_H}{N_H}}}{{{m_H}}} - {I_H}){I_V}}&{}\\
{\quad \quad \quad \quad \quad \quad \quad \quad \quad \quad \quad  - {m_H}{I_H} - {r_H}{I_H},}&{x(t) \in {\Omega _t},t \in ({{(n\tau )}^ + },(n + 1)\tau ],}\\
{\frac{{\partial {I_V}}}{{\partial t}} + {\bf{a}} \cdot \nabla {I_V} + {I_V}(\nabla  \cdot {\bf{a}}) = {d_V}\Delta {I_V} + \frac{{{\beta _V}b}}{{{N_H} + L}}(\frac{\Lambda }{{{m_V}}} - {I_V}){I_H}}&{}\\
{\quad \quad \quad \quad \quad \quad \quad \quad \quad \quad \quad  - {m_V}{I_V},}&{x(t) \in {\Omega _t},t \in ({{(n\tau )}^ + },(n + 1)\tau ],}\\
{{I_H}(x(t),{{(n\tau )}^ + }) = {I_H}(x(t),n\tau ),}&{x(t) \in {\Omega _t},}\\
{{I_V}(x(t),{{(n\tau )}^ + }) = P({I_V}(x(t),n\tau )),}&{x(t) \in {\Omega _t},n = 0,1, \cdots ,}
\end{array}} \right.
\label{a04}
\end{eqnarray}}
subject to the homogeneous Dirichlet boundary condition
\begin{eqnarray}
{I_H}(x(t),t) = {I_V}(x(t),t) = 0,\;\; x({\rm{t}}) \in \partial {\Omega _t},t > 0,
\label{a05}
\end{eqnarray}
and the initial condition
\begin{eqnarray}
{I_H}(x,0) = {I_{H,0}}(x) \le \frac{{{b_H}{N_H}}}{{{m_H}}},\;\;{I_V}(x,0) = {I_{V,0}}(x) \le \frac{\Lambda }{{{m_V}}},\;\; x \in {{\overline \Omega   }_0},
\label{a06}
\end{eqnarray}
where $\mathbf{a}$ represents the flow, the advection terms $\mathbf{a} \cdot \nabla {I_H}$ and $\mathbf{a} \cdot \nabla {I_V}$ indicate that it matches to elemental volumes transferring with the flow because of local growth, and the dilution terms ${I_H}(\nabla  \cdot \mathbf{a})$ and ${I_V}(\nabla  \cdot \mathbf{a})$ arise from the local volume expansion\cite{Baker2007}. ${\Omega _t} \subset {\mathbb{R}^N}(N \ge 1)$ is an evolving domain and $\Omega_0$ refers to the initial domain. The parameter $\tau$ denotes the period of pulsed intervention, ${(n\tau )}^ +$ represents the right-hand limit at time
$n\tau$, and $P({I_V}(x(t),n\tau ))$ stands for the impulsive function. ${I_{H,0}}(x)$ and ${I_{V,0}}(x)$ are positive continuous initial functions. Moreover, the last equation of model (\ref{a04}) indicates the spatial density of infected mosquitoes after pulsed intervention is applied at position $x(t)$ and time $n\tau$. From the third equation of the model (\ref{a04}), it follows that the infected human population is not subject to pulsed intervention.

To overcome the analytical difficulties caused by convection terms and dilution terms, we regard $\Gamma (y,t)$ as a transformation from ${{\Omega _0}}$ to ${{\Omega _t}}$. Let $t_0$ and $x_0\in\Omega_{t_0}$ be given, and define a curve $\mathbf{x}(t)=\Gamma(y_0,t)$. Then we have ${x_0} = \mathbf{x}({t_0}) = \Gamma ({y_0},{t_0})$ for ${y_0} \in {\Omega _0}$ and ${y_0} = \mathbf{X}({x_0},{t_0})$, where $\mathbf{X}$ is the inverse of $\Gamma$. Thus, it follows that $\mathbf{x}(0)=\Gamma(y_0,0)=y_0$. In addition, the flow $\mathbf{a}$ satisfies
\begin{eqnarray*}
\mathbf{a}({x_0},{t_0}) = \mathbf{x}'({t_0}) = \frac{{\partial \Gamma ({y_0},{t_0})}}{{\partial t}} = \frac{{\partial \Gamma (\mathbf{X}({x_0},{t_0}),{t_0})}}{{\partial t}}.
\end{eqnarray*}
Given the arbitrariness of $x_0$ and $t_0$, we have the general formula $
\mathbf{a}(x,t) = {\left. {\frac{{\partial \Gamma (y,t)}}{{\partial t}}} \right|_{y = \mathbf{X}(x,t)}}$ for $ x \in {\Omega _t}$ and $t \ge 0$.
For simplicity, we consider a specific type of evolving domain whose evolution follows linear isotropic deformation (see, e.g., Eq. (15) in \cite{cej1999}). Such a deformation can be expressed as
\begin{eqnarray}
x = \Gamma (y,t) = \rho (t)y,
\label{a07}
\end{eqnarray}
where $\rho (t)$ is referred to as the scaling factor and satisfies. Therefore, the domain at time $t$ is given by ${\Omega _t}: = \rho (t){\Omega _0} = \{ \rho (t)y:y \in {\Omega _0}\}$. Equation (\ref{a07}) directly yields $y = \frac{x}{{\rho (t)}}: = \mathbf{X}(x,t)$, accordingly $
\mathbf{a}(x,t) = {\left. {\dot \rho (t)y} \right|_{y = \mathbf{X}(x,t)}} = \frac{{\dot \rho (t)}}{{\rho (t)}}x$. Now, we assume that
\begin{eqnarray*}
{I_H}(x,t) = {I_H}(\rho (t)y,t): = u(y,t),\;\; {I_V}(x,t) = {I_V}(\rho (t)y,t): = v(y,t),\;\; y \in {\Omega _0},t > 0,
\end{eqnarray*}
and further obtain
$$\frac{{\partial u}}{{\partial t}} = \frac{{\partial {I_H}}}{{\partial t}} + \mathbf{a} \cdot \nabla {I_H},\;\; \frac{{\partial v}}{{\partial t}} = \frac{{\partial {I_V}}}{{\partial t}} + \mathbf{a} \cdot \nabla {I_V},$$
$$\Delta {I_H} = \frac{1}{{{\rho ^2}(t)}}\Delta u,\;\; \Delta {I_V} = \frac{1}{{{\rho ^2}(t)}}\Delta v,\;\; \nabla  \cdot \mathbf{a} = \frac{{N\dot \rho (t)}}{{\rho (t)}}.$$
We therefore transform problem (\ref{a04})-(\ref{a06}) into the following model:
{\small \begin{eqnarray}
\left\{ {\begin{array}{*{20}{l}}
{\frac{{\partial u}}{{\partial t}} = \frac{{{d_H}}}{{{\rho ^2}(t)}}\Delta u + \frac{{{\beta _H}b}}{{{N_H} + L}}(\frac{{{b_H}{N_H}}}{{{m_H}}} - u)v - (\frac{{N\dot \rho (t)}}{{\rho (t)}} + {m_H} + {r_H})u,}&{y \in {\Omega _0},t \in ({{(n\tau )}^ + },(n + 1)\tau ],}\\
{\frac{{\partial v}}{{\partial t}} = \frac{{{d_V}}}{{{\rho ^2}(t)}}\Delta v + \frac{{{\beta _V}b}}{{{N_H} + L}}(\frac{\Lambda }{{{m_V}}} - v)u - (\frac{{N\dot \rho (t)}}{{\rho (t)}} + {m_V})v,}&{y \in {\Omega _0},t \in ({{(n\tau )}^ + },(n + 1)\tau ],}\\
{u(y,{{(n\tau )}^ + }) = u(y,n\tau ),\;v(y,{{(n\tau )}^ + }) = P(v(y,n\tau )),}&{y \in {\Omega _0},}\\
{u(y,t) = v(y,t) = 0,}&{y \in \partial {\Omega _0},t > 0,}\\
{u(y,0) = {I_{H,0}}(x(0)) \le \frac{{{b_H}{N_H}}}{{{m_H}}},\;v(y,0) = {I_{V,0}}(x(0)) \le \frac{\Lambda }{{{m_V}}},}&{y \in {{\overline \Omega }_0},n = 0,1, \cdots .}
\end{array}} \right.
\label{a08}
\end{eqnarray}}

The primary objective of this paper is to explore the long-term dynamics of the solution to problem (\ref{a08}) under the effects of impulsive intervention and domain evolution. To ensure the existence and uniqueness of classical solution, we assume that the initial functions, defined by ${I_{H,0}}(x(0)):= {u_0}(y)$ and ${I_{V,0}}(x(0)) := {v_0}(y)$, satisfy ${u_0}(y), {v_0}(y) \in {C^1}(\overline{\Omega}_0)$ with ${u_0}(y), {v_0}(y) \ge 0$ (and not identically zero). We successively propose the following assumptions regarding the impulsive function $P(v)$ and the scaling factor $\rho (t)$:
\begin{enumerate}
		\item[{$\mathbf{(A1)}$:}]  $P(v)\in C^1[0,+\infty)$, $P(0)=0$, $P'(0)>0$, and $P(v)$ is nondecreasing for $v\ge0$.
		\item[{$\mathbf{(A2)}$:}]  The function $\frac{{P(v)}}{v}$ is strictly  decreasing and satisfies $0 < \frac{{P(v)}}{v} \le 1$ for $v>0$.
		\item[{$\mathbf{(H)}$:}] $\rho (t) \in {C}^1[0,\infty )$, $\rho (0) = 1$ and $\rho (t)$ is $\tau$-periodic (i.e. $\rho (t+\tau)=\rho (t))$.
\end{enumerate}

Following \cite{my2021}, several functions satisfy Assumptions $\mathbf{(A1)}$ and $\mathbf{(A2)}$. Examples include the linear function $P(v)=cv$ with $c\in (0,1]$ and the Beverton-Holt function $P(v) = \frac{{mv}}{a+v}$ with $a\in (0,\infty)$ and $m\in (0,a]$. Additionally,  in the absence of impulsive interventions (i.e., when $P(v)=v$), problem (\ref{a08}) reduces to a system involving only domain evolution, a case already discussed in \cite{zm2019}.

With the initial conditions in problem (\ref{a08}) replaced by periodic conditions, it becomes the following periodic problem:
\begin{eqnarray}
\left\{ {\begin{array}{*{20}{l}}
{\frac{{\partial U}}{{\partial t}} = \frac{{{d_H}}}{{{\rho ^2}(t)}}\Delta U + \frac{{{\beta _H}b}}{{{N_H} + L}}(\frac{{{b_H}{N_H}}}{{{m_H}}} - U)V - (\frac{{N\dot \rho (t)}}{{\rho (t)}} + {m_H} + {r_H})U,}&{y \in {\Omega _0},t \in ({0^ + },\tau ],}\\
{\frac{{\partial V}}{{\partial t}} = \frac{{{d_V}}}{{{\rho ^2}(t)}}\Delta V + \frac{{{\beta _V}b}}{{{N_H} + L}}(\frac{\Lambda }{{{m_V}}} - V)U - (\frac{{N\dot \rho (t)}}{{\rho (t)}} + {m_V})V,}&{y \in {\Omega _0},t \in ({0^ + },\tau ],}\\
{U(y,{0^ + }) = U(y,0),\;V(y,{0^ + }) = P(V(y,0)),}&{y \in {\Omega _0},}\\
{U(y,t) = V(y,t) = 0,}&{y \in \partial {\Omega _0},t \in (0,\tau ],}\\
{U(y,0) = U(y,\tau ),\;V(y,0) = V(y,\tau ),}&{y \in {{\overline \Omega }_0}.}
\end{array}} \right.
\label{a09}
\end{eqnarray}

As a crucial concept in infectious disease dynamics, the basic reproduction number $R_0$ quantifies the spread or extinction of an infectious disease. Simultaneously, the principal eigenvalue $\lambda_1$ can similarly be used as a threshold to assess the transmission of infectious diseases.

In Section 3, the existence and properties of $\lambda_1$ are discussed. The main results concerning $\lambda_1$ are presented below, with their proofs detailed in Section 4.

\begin{thm}\label{thm1.2}
If $\lambda_1 < 0$, then auxiliary problem (\ref{a09}) admits a unique positive $\tau$-periodic solution $(U(y,t), V(y,t))$.
\end{thm}

\begin{thm}\label{thm1.3}
If $\lambda_1 < 0$, then for any initial value $({u_0}(y),{v_0}(y))$ the solution $(u(y,t),v(y,t))$ to problem (\ref{a08}) fulfills
\begin{eqnarray}
\mathop {\lim }\limits_{n \to \infty } (u(y,n\tau  + t),v(y,n\tau  + t)) = ({U }(y,t),{V }(y,t)),\;\; y\in {{\overline \Omega }_0},t \in [0,\infty ).
\label{a10}
\end{eqnarray}
\end{thm}

\begin{thm}\label{thm1.1}
If ${\lambda _1} \ge 0$, then the solution $(u(y,t),v(y,t))$ to problem (\ref{a08}) satisfies $$\mathop {\lim }\limits_{t \to \infty } (u(y,t),v(y,t)) = (0,0)\;\;uniformly\;\;for\;\;y\in {{\overline \Omega }_0}.$$
\end{thm}

The remaining content is organized as follows: Section 2 establishes the well-posedness of solutions to the problem (\ref{a08}). Furthermore, to provide a more intuitive understanding of the theoretical results, several numerical simulations are presented in Section 5 to illustrate the effects of impulsive interventions and regional variations. Finally, the paper concludes with a discussion.

\section{Well-posedness}

The existence, uniqueness, and nonnegativity of the solution to problem (\ref{a08}) are demonstrated in this section.

\begin{thm}
Problem (\ref{a08}) admits a unique nonnegative solution $(u(y,t),v(y,t))$ on $\overline{\Omega}_0 \times [0, \infty)$ for any initial value $({u_0}(y),{v_0}(y))$.
\end{thm}
\textbf{Proof:} Under the conditions that $u_0(y),v_0(y) \in {C}^1(\overline{\Omega}_0)$ and $P(v)\in {C}^1[0, +\infty)$, then the new initial values satisfy $u(y, 0^+), v(y, 0^+) \in C^1(\overline{\Omega}_0)$. According to the standard theory of parabolic equations\cite{wangmingxin2021}, there exists a $\tau^* > 0$ such that problem (\ref{a08}) admits a unique local nonnegative solution $(u(y,t),v(y,t))$ on the interval $[0, \tau^*]$.

We now suppose that $(u(y,t),v(y,t))$ is the solution to problem (\ref{a08}) on $\overline{\Omega}_0 \times [0, \tau_0]$ for any given $\tau_0$. Then there exists a positive constant $M(\tau_0) $ dependent on $\tau_0$ such that $|u|,|v| \leq M(\tau_0)$. The following result presents the nonnegativity and boundedness of $u(y,t)$ and $v(y,t)$ on $\overline{\Omega}_0 \times [0, \tau_0]$. Let
\begin{eqnarray}
w(y,t) = {e^{ - kt}}u(y,t),\;\;z(y,t) = {e^{ - kt}}v(y,t),\;\;{y \in {\overline\Omega _0},t \in [0,\tau_0],}
\label{b01}
\end{eqnarray}
where $k$ is a sufficiently large constant. Substituting (\ref{b01}) into the first two equations of problem (\ref{a08}) gives that
\begin{eqnarray}
\left\{ \begin{array}{l}
{w_t} = \frac{{{d_H}}}{{{\rho ^2}(t)}}\Delta w + \frac{{{\beta _H}b}}{{{N_H} + L}}\frac{{{b_H}{N_H}}}{{{m_H}}}z + \left( { - k - \frac{{{\beta _H}b}}{{{N_H} + L}}v - \frac{{N\dot \rho (t)}}{{\rho (t)}} - {m_H} - {r_H}} \right)w,\\
{z_t} = \frac{{{d_V}}}{{{\rho ^2}(t)}}\Delta z + \frac{{{\beta _V}b}}{{{N_H} + L}}\frac{\Lambda }{{{m_V}}}w + \left( { - k - \frac{{{\beta _V}b}}{{{N_H} + L}}u - \frac{{N\dot \rho (t)}}{{\rho (t)}} - {m_V}} \right)z.
\end{array} \right.
\label{b02}
\end{eqnarray}
We can thus prove that $ (w, z) \geq (0, 0)$ for $(y, t) \in \overline{\Omega}_0 \times [0, \tau_0]$. If not, then $\min (w, z) < 0$. Without loss of generality, assume that $\min(w, z) = w(y_0, t_0) < 0$ for $(y_0, t_0)\in \overline{\Omega}_0 \times [0, \tau_0]$. It follows that this is in contradiction with the first line of (\ref{b02}). In fact, the left-hand side of the first line satisfies $w_t(y_0, t_0) \le 0$, whereas the first term on the right-hand side satisfies $ \Delta w(y_0, t_0) \geq 0$. Besides, given that $w(y_0, t_0) \leq z(y_0, t_0)$ and the selection of $k$, one has
$$\frac{{{\beta _H}b}}{{{N_H} + L}}\frac{{{b_H}{N_H}}}{{{m_H}}}z({y_0},{t_0}) + \left( { - k - \frac{{{\beta _H}b}}{{{N_H} + L}}v - \frac{{N\dot \rho (t)}}{{\rho (t)}} - {m_H} - {r_H}} \right)w({y_0},{t_0}) > 0.$$
Therefore, $(w, z) \geq (0, 0)$ holds for $(y, t) \in \overline{\Omega}_0 \times [0, \tau_0]$. In view of (\ref{b01}), it follows that $(u(y,t),v(y,t)) \geq (0, 0)$ holds for $(y, t) \in \overline{\Omega}_0 \times [0, \tau_0]$. Furthermore, from problem (\ref{a08}) we obtain $(u(y,t), v(y,t)) \leq \left( \frac{b_H N_H}{m_H}, \frac{\Lambda}{m_V} \right) $ for $(y, t) \in \overline{\Omega}_0 \times [0, \tau_0]$.

Let $[0,\tau_{max})$ be the maximal interval of existence of the solution, with $\tau_{max}>0$. We prove that $\tau_{max}=\infty$. Otherwise, assume $\tau_{max}<\infty$. If there exists an integer $n_0$ such that $\tau_{max}\neq n_0\tau$, then the above result indicates that $u(y,t),v(y,t) \le \max \{ {\frac{{{b_H}{N_H}}}{{{m_H}}},\frac{\Lambda }{{{m_V}}}}\} := {C_0}$ for $(y, t) \in \overline{\Omega}_0 \times [0, \tau_{\text{max}})$. Now fix $\varepsilon_0 \in (0, \tau_{\text{max}})$ and $M_0 > \tau_{\text{max}}$. According to the $L^p$ theory, Sobolev embedding theorem and H$\ddot{o}$lder estimate\cite{Ladyzhenskaya1968}, there exists a $C_1$ depending only on $\varepsilon_0$, $M_0$ and $C_0$ such that ${\left\| {u( \cdot ,t)} \right\|_{{C^1}({{\overline \Omega }_0})}},{\left\| {v( \cdot ,t)} \right\|_{{C^1}({{\overline \Omega }_0})}} \le {C_1}$ for $t \in [{\varepsilon _0},{\tau _{\max }})$. Then there exists a $\bar \tau  > 0$ depending only on $C_0$ and $C_1$ such that the solution to problem (\ref{a08}) with ${\tau _{\max }} - \bar \tau /2$ as the initial time can be uniquely extended to ${\tau _{\max }} + \bar \tau /2$, which contradicts the assumption. If $\tau_{max}= n_0\tau$, we can replace the initial time with $(n_0\tau)^+$ and then apply the same method to obtain the above result.

Since the right-hand side of problem (1.6) satisfies the Lipschitz condition, problem (\ref{a08}) possesses a unique solution $(u(y,t),v(y,t))$ on $[0,\tau]$. Using $u(y, \tau^+)$ and $(v(y, \tau^+)$ as the new initial values of $u(y,\tau)$ and $v(y,\tau)$ in $[\tau,2\tau]$, we note that $u(y, \tau^+),v(y, \tau^+) \in C^1(\overline{\Omega}_0)$. Following the same process as before, the solution $(u(y,t),v(y,t))$ to problem (\ref{a08}) on $[\tau,2\tau]$ is also unique. By proceeding step by step, we conclude that problem (\ref{a08}) admits a globally unique nonnegative solution $(u(y,t),v(y,t))$ on $[0, \infty)$. The proof is hereby completed.\epf
\section{The existence and properties of the principal eigenvalue}
Upon linearizing the periodic problem (\ref{a09}) at the equilibrium $(0, 0)$, the corresponding periodic eigenvalue problem under consideration is given below:

\begin{eqnarray}
\left\{ {\begin{array}{*{20}{l}}
{\frac{{\partial {\phi _1}}}{{\partial t}} = \frac{{{d_H}}}{{{\rho ^2}(t)}}\Delta {\phi _1} + \frac{{{\beta _H}b}}{{{N_H} + L}}\frac{{{b_H}{N_H}}}{{{m_H}}}{\phi _2} - (\frac{{N\dot \rho (t)}}{{\rho (t)}} + {m_H} + {r_H}){\phi _1} + \lambda_1 {\phi _1},}&{y \in {\Omega _0},t \in ({0^ + },\tau ],}\\
{\frac{{\partial {\phi _2}}}{{\partial t}} = \frac{{{d_V}}}{{{\rho ^2}(t)}}\Delta {\phi _2} + \frac{{{\beta _V}b}}{{{N_H} + L}}\frac{\Lambda }{{{m_V}}}{\phi _1} - (\frac{{N\dot \rho (t)}}{{\rho (t)}} + {m_V}){\phi _2} + \lambda_1 {\phi _2},}&{y \in {\Omega _0},t \in ({0^ + },\tau ],}\\
{{\phi _1}(y,{0^ + }) = {\phi _1}(y,0),\;{\phi _2}(y,{0^ + }) = P'(0){\phi _2}(y,0),}&{y \in {\Omega _0},}\\
{{\phi _1}(y,t) = {\phi _2}(y,t) = 0,}&{y \in \partial {\Omega _0},t \in [0,\tau ],}\\
{{\phi _1}(y,0) = {\phi _1}(y,\tau ),\;{\phi _2}(y,0) = {\phi _2}(y,\tau ),}&{y \in {{\overline \Omega }_0}.}
\end{array}} \right.
\label{c01}
\end{eqnarray}

We only need to take $({\phi _1}(y,t),{\phi _2}(y,t)) = ({\psi _1}(y,t),{\psi _2}(y,t))$ in $\overline\Omega_0 \times [0, \tau]$ and ${\phi _2}(y,0) = {\phi _2}(y,\tau ) = {\psi _2}(y,\tau ) = 1/P'(0){\psi _2}(y,0)$ in $\overline\Omega_0$, which enables us to obtain the equivalent eigenvalue problem to problem (\ref{c01}) as follows:
\begin{eqnarray}
\left\{ {\begin{array}{*{20}{l}}
{\frac{{\partial {\psi _1}}}{{\partial t}} = \frac{{{d_H}}}{{{\rho ^2}(t)}}\Delta {\psi _1} + \frac{{{\beta _H}b}}{{{N_H} + L}}\frac{{{b_H}{N_H}}}{{{m_H}}}{\psi _2} - (\frac{{N\dot \rho (t)}}{{\rho (t)}} + {m_H} + {r_H}){\psi _1} + {\lambda _1}{\psi _1},}&{y \in {\Omega _0},t \in ({0 },\tau ],}\\
{\frac{{\partial {\psi _2}}}{{\partial t}} = \frac{{{d_V}}}{{{\rho ^2}(t)}}\Delta {\psi _2} + \frac{{{\beta _V}b}}{{{N_H} + L}}\frac{\Lambda }{{{m_V}}}{\psi _1} - (\frac{{N\dot \rho (t)}}{{\rho (t)}} + {m_V}){\psi _2} + \lambda_1 {\psi _2},}&{y \in {\Omega _0},t \in ({0 },\tau ],}\\
{{\psi _1}(y,0) = {\psi _1}(y,\tau ),\;{\psi _2}(y,0) = P'(0){\psi _2}(y,\tau ),}&{y \in {{\overline \Omega }_0},}\\
{{\psi _1}(y,t) = {\psi _2}(y,t) = 0,}&{y \in \partial {\Omega _0},t \in [0,\tau ].}\\
\end{array}} \right.
\label{c02}
\end{eqnarray}
The eigenvalue problem (\ref{c02}) is introduced primarily to address the analytical difficulties arising from impulses. Through comparison of problems (\ref{c01}) and (\ref{c02}), it can be found that ${\phi _1}(y,{0^ + }) = {\psi _1}(y,0)$ and ${\phi _2}(y,{0^ + }) = P'(0){\phi _2}(y,0) = {\psi _2}(y,0)$ for $y\in\overline\Omega_0$.

The result presented below is related to the existence of the principal eigenvalue for problem (\ref{c01}).
\begin{thm}\label{thm3.1} Problem (\ref{c01}) possesses a principal eigenvalue $\lambda_1$, with the corresponding eigenfunctions satisfying $(\phi_1(y,t), \phi_2(y,t)) \gg (0,0)$ in $\Omega_0 \times [0, \tau]$.
\end{thm}

Before proving Theorem \ref{thm3.1}, we introduce some notations and preliminary content. Define
$$\begin{array}{l}
{Y} \triangleq \{ G := ({\psi _1},{\psi _2}) \in {[{{C}^{1,0}}({{\overline \Omega }_0} \times [0,\tau ])]^2}\,|\,G = (0,0)\;\forall (y,t) \in \partial {\Omega _0} \times [0,\tau ],\\
\quad \quad {\psi _1}(y,0) = {\psi _1}(y,\tau ),\;{\psi _2}(y,0) = P'(0){\psi _2}(y,\tau )\;\forall y \in {{\overline \Omega }_0}\};
\end{array}$$
$${{Y}^ + } \triangleq \texttt{closure}\{ G \in {Y}\,|\,G(y,t) \gg 0\;\forall (y,t) \in {\Omega _0} \times [0,\tau ],\;\frac{{\partial G}}{{\partial \nu }}(y,t) \ll 0\;\forall (y,t) \in \partial {\Omega _0} \times [0,\tau ]\},$$ and
$$\texttt{Int}({{Y}^ + }) \triangleq \{ G \in {Y}\,|\,G(y,t) \gg 0\;\forall (y,t) \in {\Omega _0} \times [0,\tau ],\;\frac{{\partial G}}{{\partial \nu }}(y,t) \ll 0\;\forall (y,t) \in \partial {\Omega _0} \times [0,\tau ]\} ,$$
where ${Y}$ is a Banach space, ${Y}^+$ stands for the positive cone of ${Y}$, Int$({Y}^+)$ represents the interior of ${Y}^+$, and $\nu$ the outward unit normal vector of $\partial \Omega_0$. Besides, Int$({Y}^+)$ is nonempty, as well as $G(y,t) \gg 0$ and $\frac{\partial G}{\partial \nu}(y,t) \ll 0$ indicate $\psi_1 > 0, \psi_2 > 0$ and $\frac{\partial \psi_1}{\partial \nu} < 0, \frac{\partial \psi_2}{\partial \nu} < 0$, respectively.

With respect to any given vector $G(y,t)\in {Y}$, we consider the auxiliary problem given below:
\begin{eqnarray}
\left\{ {\begin{array}{*{20}{l}}
{\frac{{\partial {\phi _1}}}{{\partial t}} = \frac{{{d_H}}}{{{\rho ^2}(t)}}\Delta {\phi _1} + \frac{{{\beta _H}b}}{{{N_H} + L}}\frac{{{b_H}{N_H}}}{{{m_H}}}{\phi _2} - (\frac{{N\dot \rho (t)}}{{\rho (t)}} + {m_H} + {r_H} + \overline M){\phi _1} + {\psi _1},}&{y \in {\Omega _0},t \in (0,\tau ],}\\
{\frac{{\partial {\phi _2}}}{{\partial t}} = \frac{{{d_V}}}{{{\rho ^2}(t)}}\Delta {\phi _2} + \frac{{{\beta _V}b}}{{{N_H} + L}}\frac{\Lambda }{{{m_V}}}{\phi _1} - (\frac{{N\dot \rho (t)}}{{\rho (t)}} + {m_V} + \overline M){\phi _2} + {\psi _2},}&{y \in {\Omega _0},t \in (0,\tau ],}\\
{{\phi _1}(y,0) = {\phi _1}(y,\tau ),\;{\phi _2}(y,0) = P'(0){\phi _2}(y,\tau ),}&{y \in {{\overline \Omega }_0},}\\
{{\phi _1}(y,t) = {\phi _2}(y,t) = 0,}&{y \in \partial {\Omega _0},t \in [0,\tau ],}
\end{array}} \right.
\label{c03}
\end{eqnarray}
where $\overline M = {-\min _{t \in [0,\tau ]}}\frac{{N\dot \rho (t)}}{{\rho (t)}} + \frac{{{\beta _V}b}}{{{N_H} + L}}\frac{\Lambda }{{{m_V}}}\max \{ 1,P'(0)\}  + \frac{{{\beta _H}b}}{{{N_H} + L}}\frac{{{b_H}{N_H}}}{{{m_H}}}$.
\begin{lem}\label{lem3.2} For any $G(y,t)\in {Y}$, problem (\ref{c03}) admits a unique solution $H:=({\phi _1}(y,t),{\phi _2}(y,t)) \in {Y} \cap {{C}^{1 + \alpha ,(1 + \alpha )/2}}({{\overline \Omega }_0} \times [0,\tau ])$ with any $0<\alpha<1$.
\end{lem}
\textbf{Proof:} Set
$$\psi _1^* = \mathop {\max }\limits_{{y \in {{\overline \Omega }_0},t \in [0,\tau ]}} {\psi _1},\;\;\psi _2^* = \mathop {\max }\limits_{y \in {{\overline \Omega }_0},t \in [0,\tau ]} {\psi _2},\;\;\check{C} = (\psi _1^* + \psi _2^*)/2,$$
and define
$${W} \triangleq \{ ({w_1}(y),{w_2}(y)) \in {[{C}({{\overline \Omega }_0})]^2}\,|\,\left| {{w_1}(y)} \right| \le \check{C},\left| {{w_2}(y)} \right| \le \check{C},\;{w_1}(y) = {w_2}(y) = 0\;\forall y \in \partial {\Omega _0}\} .$$
By embedding theorem and classical $L^p$ theory\cite{Ladyzhenskaya1968}, the solution $(\eta_1(y,t), \eta_2(y, t))$ to the following initial-boundary value problem:
\begin{eqnarray}
\left\{ {\begin{array}{*{20}{l}}
{\frac{{\partial {\eta _1}}}{{\partial t}} = \frac{{{d_H}}}{{{\rho ^2}(t)}}\Delta {\eta _1} + \frac{{{\beta _H}b}}{{{N_H} + L}}\frac{{{b_H}{N_H}}}{{{m_H}}}{\eta _2} - (\frac{{N\dot \rho (t)}}{{\rho (t)}} + {m_H} + {r_H} + \overline M){\eta _1} + {\psi _1},}&{y \in {\Omega _0},t \in (0,\tau ],}\\
{\frac{{\partial {\eta _2}}}{{\partial t}} = \frac{{{d_V}}}{{{\rho ^2}(t)}}\Delta {\eta _2} + \frac{{{\beta _V}b}}{{{N_H} + L}}\frac{\Lambda }{{{m_V}}}{\eta _1} - (\frac{{N\dot \rho (t)}}{{\rho (t)}} + {m_V} + \overline M){\eta _2} + {\psi _2},}&{y \in {\Omega _0},t \in (0,\tau ],}\\
{{\eta _1}(y,0) = {w_1}(y),\;{\eta _2}(y,0) = {w_2}(y),}&{y \in {{\overline \Omega }_0},}\\
{{\eta _1}(y,t) = {\eta _2}(y,t) = 0,}&{y \in \partial {\Omega _0},t \in [0,\tau ]}
\end{array}} \right.
\label{c04}
\end{eqnarray}
satisfies $({\eta _1}(y,t),{\eta _2}(y,t)) \in {{Y}_p^{2,1}({\Omega _0} \times [0,\tau ]) \cap {{C}^{1 + \alpha ,(1 + \alpha )/2}}({{\overline \Omega }_0} \times [0,\tau ])}$ for any $p>0$ and $0<\alpha<1$. Given any arbitrary vector $({w_1}(y),{w_2}(y)) \in{W}$, we define $$\mathcal{T}(({w_1}(y),{w_2}(y)))=(\eta_1(y,\tau),P'(0)\eta_2(y,\tau)).$$
According to the comparison principle and the selection of $\overline M$, it can be readily concluded that $\mathcal{T}(({w_1}(y),{w_2}(y)))\in {W}$ for $({w_1}(y),{w_2}(y))\in {W}$. The Sobolev embedding theorem indicates that $\mathcal{T}$ is compact, which, combined with the fixed point theory, ensures that $\mathcal{T}$ possesses at least one fixed point $({w_1}(y),{w_2}(y))\in {W}$, and a solution $(\eta_1, \eta_2)(y, t)$ to problem (\ref{c04}) with $({w_1}(y),{w_2}(y))$ as the initial value. Thus,  $(\eta_1, \eta_2)(y, t)$ is the solution to problem (\ref{c03}), and $({\eta _1},{\eta _2}) \in {[{Y} \cap {{C}^{1 + \alpha ,(1 + \alpha )/2}}({{\overline \Omega }_0} \times [0,\tau ])]^2}$.

We now proceed to prove the uniqueness of the solution for problem (\ref{c03}). Actually, we merely need to prove that the homogeneous linear problem of problem (\ref{c03}) has only the trivial solution (0,0), as follows:
\begin{eqnarray}
\left\{ {\begin{array}{*{20}{l}}
{\frac{{\partial {\phi _1}}}{{\partial t}} = \frac{{{d_H}}}{{{\rho ^2}(t)}}\Delta {\phi _1} + \frac{{{\beta _H}b}}{{{N_H} + L}}\frac{{{b_H}{N_H}}}{{{m_H}}}{\phi _2} - (\frac{{N\dot \rho (t)}}{{\rho (t)}} + {m_H} + {r_H} + \overline M){\phi _1},}&{y \in {\Omega _0},t \in (0,\tau ],}\\
{\frac{{\partial {\phi _2}}}{{\partial t}} = \frac{{{d_V}}}{{{\rho ^2}(t)}}\Delta {\phi _2} + \frac{{{\beta _V}b}}{{{N_H} + L}}\frac{\Lambda }{{{m_V}}}{\phi _1} - (\frac{{N\dot \rho (t)}}{{\rho (t)}} + {m_V} + \overline M){\phi _2},}&{y \in {\Omega _0},t \in (0,\tau ],}\\
{{\phi _1}(y,0) = {\phi _1}(y,\tau ),\;{\phi _2}(y,0) = P'(0){\phi _2}(y,\tau ),}&{y \in {{\overline \Omega }_0},}\\
{{\phi _1}(y,t) = {\phi _2}(y,t) = 0,}&{y \in \partial {\Omega _0},t \in [0,\tau ],}
\end{array}} \right.
\label{c05}
\end{eqnarray}
By contradiction, suppose that
$$\max \{ \mathop {\max }\limits_{y \in {{\overline \Omega }_0},t \in [0,\tau ]} {\phi _1}(y,t),\mathop {\max }\limits_{y \in {{\overline \Omega }_0},t \in [0,\tau ]} {\phi _2}(y,t)\}  > 0\;\;or\;\;\min \{ \mathop {\min }\limits_{y \in {{\overline \Omega }_0},t \in [0,\tau ]} {\phi _1}(y,t),\mathop {\min }\limits_{y \in {{\overline \Omega }_0},t \in [0,\tau ]} {\phi _2}(y,t)\}  < 0.$$
Without loss of generality, we presume that there exists a point $(y_0, t_0)\in {\overline \Omega }_0\times [0,\tau ]$ so that $$\max \{ {\max _{y \in {{\overline \Omega }_0},t \in [0,\tau ]}}{\phi _1}(y,t),{\max _{y \in {{\overline \Omega }_0},t \in [0,\tau ]}}{\phi _2}(y,t)\}  = {\phi _2}({y_0},{t_0}) > 0.$$
In the second equation of problem (\ref{c05}), it shows that
\begin{eqnarray}
\frac{{\partial {\phi _2}({y_0},\tau )}}{{\partial t}} - \frac{{{d_V}}}{{{\rho ^2}(t)}}\Delta {\phi _2}({y_0},\tau ) = \frac{{{\beta _V}b}}{{{N_H} + L}}\frac{\Lambda }{{{m_V}}}{\phi _1}({y_0},\tau ) - (\frac{{N\dot \rho (t)}}{{\rho (t)}} + {m_V} + \overline M){\phi _2}({y_0},\tau ).
\label{c06}
\end{eqnarray}
When $t_0=0$, from (\ref{c06}) and the boundary conditions for problem (\ref{c05}), it follows that
\begin{eqnarray}
\frac{{\partial {\phi _2}({y_0},0)}}{{\partial t}} - \frac{{{d_V}}}{{{\rho ^2}(t)}}\Delta {\phi _2}({y_0},0) = \frac{{{\beta _V}b}}{{{N_H} + L}}\frac{\Lambda }{{{m_V}}}P'(0){\phi _1}({y_0},0) - (\frac{{N\dot \rho (t)}}{{\rho (t)}} + {m_V} + \overline M){\phi _2}({y_0},0).
\label{c07}
\end{eqnarray}
Evidently, the left-hand side of (\ref{c06}) is nonnegative, but the right-hand side is negative due to $\overline M > {-\min _{t \in [0,\tau ]}}\frac{{N\dot \rho (t)}}{{\rho (t)}} + \frac{{{\beta _V}b}}{{{N_H} + L}}\frac{\Lambda }{{{m_V}}}\max \{ 1,P'(0)\} $, which leads to a contradiction. When $t_0 \in (0,\tau ]$, we have
$$\frac{{\partial {\phi _2}({y_0},{t_0})}}{{\partial t}} - \frac{{{d_V}}}{{{\rho ^2}(t)}}\Delta {\phi _2}({y_0},{t_0}) \ge 0,\;\;\frac{{{\beta _V}b}}{{{N_H} + L}}\frac{\Lambda }{{{m_V}}}{\phi _1}({y_0},{t_0}) - (\frac{{N\dot \rho (t)}}{{\rho (t)}} + {m_V} + \overline M){\phi _2}({y_0},{t_0}) < 0,$$
which stands in contradiction to (\ref{c06}). Thus, the proof is completed.\epf\vspace{0.5\baselineskip}

Now, we present the proof of Theorem \ref{thm3.1} based on Lemma \ref{lem3.2}.\vspace{0.5\baselineskip}

\noindent\textbf{Proof of Theorem \ref{thm3.1}:} We define an operator $\mathcal{A }G = H$, where $G$ and $H$ are the solutions to problems (\ref{c02}) and (\ref{c03}), respectively. This operator $\mathcal{A}$ is strongly positive with respect to ${Y}$ as well as a linear compact operator. The former follows from the strong maximum principle and Hopf boundary lemma\cite{Protter1984}, and the latter is due to the fact that the embedding ${C^{1 + \alpha ,(1 + \alpha )/2}}\hookrightarrow{C^{1,0}}$ is compact.

It follows from the strong version of the Kerin-Rutman theorem\cite{Krein1950} that $\mathbf{r}(\mathcal{A})>0$ is an algebraically simple eigenvalue of $\mathcal{A}$ with a corresponding eigenvector $H \in \texttt{Int}({Y}^+)$, and there exists no other eigenvalue that admits a positive eigenvector.

Accordingly, ${\lambda _1} = 1/\mathbf{r}(\mathcal{A}) - \overline M$ is the eigenvalue of problem \ref{c01}, and the corresponding eigenfunction satisfies $H=({\phi _1}(y,t),{\phi _2}(y,t))\gg(0,0)$ in ${{\Omega _0}}\times{[0,\tau ]}$.\epf\vspace{0.5\baselineskip}

Analogue to Theorem 5.2 in \cite{zq2025}, the following properties of the principal eigenvalue ${\lambda _1}$ are given.

\begin{lem}\label{thm3.4} The principal eigenvalue $\lambda_1$ is strictly monotonically decreasing with respect to $P'(0)$ and $\Omega_0$, respectively.
\end{lem}

The specific role of Theorem \ref{thm3.4} is embodied in the numerical simulations, and the detailed proof can be found in Section 5 of  \cite{zq2025}. We omit here.

\section{The dynamical behaviors of the solution}

In this section, we study the long-time behavior of the solution to problem (\ref{a08}) with respect to the principal eigenvalue $\lambda_1$ using the upper-lower solution method and the comparison principle.

The definitions of upper and lower solutions for the initial problem (\ref{a08}) and the periodic problem (\ref{a09}) are presented. Define
$$PC \triangleq PC({{\overline \Omega }_0} \times [0,\infty )) \triangleq \{ (u,v)|u,v \in { \cap _{n \in N}}C({{\overline \Omega }_0} \times (n\tau ,(n + 1)\tau ])\}$$
and
$$P{C^{2,1}} \triangleq P{C^{2,1}}({\Omega _0} \times (0,\infty )) \triangleq \{ (u,v)|u,v \in { \cap _{n \in N}}{C^{2,1}}({\Omega _0} \times (n\tau ,(n + 1)\tau ])\} .$$
\begin{defi}\label{defi4.1}Assume that $(\tilde{u}, \tilde{v}),(\hat{u}, \hat{v}) \in PC^{2,1} \cap PC$. If $(0,0) \le (\hat{u}, \hat{v}) \le (\tilde{u}, \tilde{v})$ holds and the following conditions satisfy

\begin{eqnarray}
\left\{ {\begin{array}{*{20}{l}}
{\frac{{\partial \tilde u}}{{\partial t}} \ge \frac{{{d_H}}}{{{\rho ^2}(t)}}\Delta \tilde u + \frac{{{\beta _H}b}}{{{N_H} + L}}(\frac{{{b_H}{N_H}}}{{{m_H}}} - \tilde u)\tilde v - (\frac{{N\dot \rho (t)}}{{\rho (t)}} + {m_H} + {r_H})\tilde u,}&{y \in {\Omega _0},t \in ({{(n\tau )}^ + },(n + 1)\tau ],}\\
{\frac{{\partial \hat u}}{{\partial t}} \le \frac{{{d_H}}}{{{\rho ^2}(t)}}\Delta \hat u + \frac{{{\beta _H}b}}{{{N_H} + L}}(\frac{{{b_H}{N_H}}}{{{m_H}}} - \hat u)\hat v - (\frac{{N\dot \rho (t)}}{{\rho (t)}} + {m_H} + {r_H})\hat u,}&{y \in {\Omega _0},t \in ({{(n\tau )}^ + },(n + 1)\tau ],}\\
{\frac{{\partial \tilde v}}{{\partial t}} \ge \frac{{{d_V}}}{{{\rho ^2}(t)}}\Delta \tilde v + \frac{{{\beta _V}b}}{{{N_H} + L}}(\frac{\Lambda }{{{m_V}}} - \tilde v)\tilde u - (\frac{{N\dot \rho (t)}}{{\rho (t)}} + {m_V})\tilde v,}&{y \in {\Omega _0},t \in ({{(n\tau )}^ + },(n + 1)\tau ],}\\
{\frac{{\partial \hat v}}{{\partial t}} \le \frac{{{d_V}}}{{{\rho ^2}(t)}}\Delta \hat v + \frac{{{\beta _V}b}}{{{N_H} + L}}(\frac{\Lambda }{{{m_V}}} - \hat v)\hat u - (\frac{{N\dot \rho (t)}}{{\rho (t)}} + {m_V})\hat v,}&{y \in {\Omega _0},t \in ({{(n\tau )}^ + },(n + 1)\tau ],}\\
{\tilde u(y,{{(n\tau )}^ + }) \ge \tilde u(y,n\tau ),\;\tilde v(y,{{(n\tau )}^ + }) \ge P(\tilde v(y,n\tau )),}&{y \in {\Omega _0},}\\
{\hat u(y,{{(n\tau )}^ + }) \le \hat u(y,n\tau ),\;\hat v(y,{{(n\tau )}^ + }) \le P(\hat v(y,n\tau )),}&{y \in {\Omega _0},}\\
{\hat u(y,t) = 0 \le \tilde u(y,t),\;\hat v(y,t) = 0 \le \tilde v(y,t),}&{y \in \partial {\Omega _0},t > 0,n = 0,1, \cdots, }
\end{array}} \right.
\label{d01}
\end{eqnarray}
and
\begin{eqnarray}
\left\{ {\begin{array}{*{20}{l}}
{\tilde u(y,0) \ge {u_0}(y),\;\tilde v(y,0) \ge {v_0}(y),}&{y \in {{\overline \Omega }_0},}\\
{\hat u(y,0) \le {u_0}(y),\;\hat v(y,0) \le {v_0}(y),}&{y \in {{\overline \Omega }_0}},
\end{array}} \right.
\label{d02}
\end{eqnarray}
then $(\tilde{u}, \tilde{v})$ and $(\hat{u}, \hat{v})$ are called the ordered upper and lower solutions to problem (\ref{a08}).
\end{defi}
\begin{defi}\label{defi4.2} Let $n=0$ in (\ref{d01}) and change the variables from $(\tilde{u}, \tilde{v}),(\hat{u}, \hat{v})$ to $(\widetilde{U}, \widetilde{V}),(\widehat{U}, \widehat{V})$. Suppose that $(\widetilde{U}, \widetilde{V}),(\widehat{U}, \widehat{V}) \in PC^{2,1} \cap PC$. If $(0,0) \le (\widehat{U}, \widehat{V}) \le (\widetilde{U}, \widetilde{V})$ holds and condition (\ref{d02}) is replaced by
\begin{eqnarray*}
\left\{ {\begin{array}{*{20}{l}}
{\widetilde U(y,0) \ge \widetilde U(y,\tau ),\;\widetilde V(y,0) \ge \widetilde V(y,\tau ),}&{y \in {{\overline \Omega }_0},}\\
{\widehat U(y,0) \le \widehat U(y,\tau ),\;\widehat V(y,0) \le \widehat V(y,\tau ),}&{y \in {{\overline \Omega }_0},}
\end{array}} \right.
\end{eqnarray*}
then $(\widetilde{U}, \widetilde{V})$ and $(\widetilde{U}, \widetilde{V})$ are called the ordered upper and lower solutions to problem (\ref{a09}).
\end{defi}

Combined with Lemma 3.1 of \cite{my2021}, the comparison principle for problem (\ref{a08}) with impulses is presented below.

\begin{lem}\label{lem4.1} Provided that $(\tilde{u}(y,t ), \tilde{v}(y,t ))$ and $(\hat{u}(y,t ), \hat{v}(y,t ))$ are the ordered upper-lower solutions to problem (\ref{a08}), the unique solution $(u(y,t ), v(y,t ))$ to problem (\ref{a08}) is found to satisfy
$$(\hat u(y,t),\hat v(y,t)) \le (u(y,t),v(y,t)) \le (\tilde u(y,t),\tilde v(y,t)),\;\; y \in {\Omega _0},t \ge 0.$$
\end{lem}

In the following, we study the long-term dynamics of solutions to problem (\ref{a08}) when $\lambda_1 < 0$. For this purpose, we first prove Theorem \ref{thm1.2}, which is partitioned into two parts: one dedicated to the existence of $\tau$-periodic solutions, and the other to their uniqueness.\vspace{0.5\baselineskip}

\noindent\textbf{Proof of Theorem \ref{thm1.2}:} \textbf{Step 1.} We set \begin{eqnarray}
(\overline U(y,t),\overline V(y,t)) = (\overline C,\overline C),
\label{d05}
\end{eqnarray}
where $$\overline C \ge \max \{ \frac{{ - \mathop {\min }\limits_{t \in [0,\tau ]} \frac{{N\dot \rho (t)}}{{\rho (t)}} + \frac{{{\beta _H}b}}{{{N_H} + L}}\frac{{{b_H}{N_H}}}{{{m_H}}}}}{{\frac{{{\beta _H}b}}{{{N_H} + L}}}},\frac{{ - \mathop {\min }\limits_{t \in [0,\tau ]} \frac{{N\dot \rho (t)}}{{\rho (t)}} + \frac{{{\beta _V}b}}{{{N_H} + L}}\frac{\Lambda }{{{m_V}}}}}{{\frac{{{\beta _V}b}}{{{N_H} + L}}}},1\}.$$
The first and third equations in (\ref{d01}) after the substitution, periodic condition $(\overline U(y,0),\overline V(y,0)) = (\overline U(y,\tau ),\overline V(y,\tau ))$ for $y \in {\overline\Omega _0}$, and the boundary condition $(\overline{U}(y,t), \overline{V}(y,t))>(0,0)$ for $y \in \partial {\Omega _0},t>0$ hold. Besides, it follows that
$$\overline U(y,{0^ + }) - \overline U(y,0)=0$$
and
$$\overline V(y,{0^ + }) - P(\overline V(y,0)) = \overline C - P(\overline C) \ge 0$$
due to Assumption $\mathbf{(A2)}$. Therefore, according to Definition \ref{defi4.2}, $(\overline{U}(y,t), \overline{V}(y,t))$ is an upper solution to problem (\ref{a09}).

Next, a lower solution to problem (\ref{a09}) is constructed. Define
\begin{eqnarray*}
\underline U (y,t) = \left\{ {\begin{array}{*{20}{l}}
{\varepsilon {\phi _1}(y,t),}&{y \in {{\overline \Omega }_0},t = 0,}\\
{\varepsilon {e^{({\lambda _1} + \mu )\tau}}{\phi _1}(y,t),}&{y \in {{\overline \Omega }_0},t = {0^ + },}\\
{\varepsilon {e^{ - ({\lambda _1} + \mu )(t - \tau )}}{\phi _1}(y,t),}&{y \in {{\overline \Omega }_0},t = ({0^ + },\tau ]}
\end{array}} \right.
\label{d05'}
\end{eqnarray*}
and
\begin{eqnarray*}
\underline V (y,t) = \left\{ {\begin{array}{*{20}{l}}
{\varepsilon {\phi _2}(y,t),}&{y \in {{\overline \Omega }_0},t = 0,}\\
{\varepsilon {e^{({\lambda _1} + \mu )\tau}}{\phi _2}(y,t),}&{y \in {{\overline \Omega }_0},t = {0^ + },}\\
{\varepsilon {e^{ - ({\lambda _1} + \mu )(t - \tau )}}{\phi _2}(y,t),}&{y \in {{\overline \Omega }_0},t = ({0^ + },\tau ]},
\end{array}} \right.
\label{d05''}
\end{eqnarray*}
where $(\phi_1(y,t), \phi_2(y,t))$ stands for a pair of eigenfunctions corresponding to problem (\ref{c01}), and $\varepsilon$ is a sufficiently small positive constant so that
$$\underline V (y,{0^ + }) - P(\underline V (y,0)) = \varepsilon {e^{({\lambda _1} + \mu )\tau }}P'(0){\phi _2}(y,0) - P(\varepsilon {\phi _2}(y,0)) \le 0.$$
In addition, $\mu$ is a positive constant such that $\lambda_1 + \mu < 0$ for the given eigenvalue $ \lambda_1$. By problem (\ref{c01}), we have
$$\begin{array}{l}
\frac{{\partial \underline U }}{{\partial t}} - \frac{{{d_H}}}{{{\rho ^2}(t)}}\Delta \underline U  - \frac{{{\beta _H}b}}{{{N_H} + L}}(\frac{{{b_H}{N_H}}}{{{m_H}}} - \underline U )\underline V  + (\frac{{N\dot \rho (t)}}{{\rho (t)}} + {m_H} + {r_H})\underline U \\
 = \varepsilon {e^{ - ({\lambda _1} + \mu )(t - \tau )}}(\frac{{\partial {\phi _1}}}{{\partial t}} - ({\lambda _1} + \mu ){\phi _1}- \frac{{{d_H}}}{{{\rho ^2}(t)}}\Delta {\phi _1}\\
  \quad - \frac{{{\beta _H}b}}{{{N_H} + L}}(\frac{{{b_H}{N_H}}}{{{m_H}}} - \varepsilon {e^{ - ({\lambda _1} + \mu )(t - \tau )}}{\phi _1}){\phi _2} + (\frac{{N\dot \rho (t)}}{{\rho (t)}} + {m_H} + {r_H}){\phi _1})\\
 = \varepsilon {e^{ - ({\lambda _1} + \mu )(t - \tau )}}( - \mu {\phi _1} + \varepsilon {e^{ - ({\lambda _1} + \mu )(t - \tau )}}\frac{{{\beta _H}b}}{{{N_H} + L}}{\phi _1}{\phi _2})\\
 \le 0
\end{array}$$
and
$$\begin{array}{l}
\frac{{\partial \underline V }}{{\partial t}} - \frac{{{d_V}}}{{{\rho ^2}(t)}}\Delta \underline V  - \frac{{{\beta _V}b}}{{{N_H} + L}}(\frac{\Lambda }{{{m_V}}} - \underline V )\underline U + (\frac{{N\dot \rho (t)}}{{\rho (t)}} + {m_V})\underline V \\
 = \varepsilon {e^{ - ({\lambda _1} + \mu )(t - \tau )}}(\frac{{\partial {\phi _2}}}{{\partial t}} - ({\lambda _1} + \mu ){\phi _2} - \frac{{{d_V}}}{{{\rho ^2}(t)}}\Delta {\phi _2}\\
  \quad- \frac{{{\beta _V}b}}{{{N_H} + L}}(\frac{\Lambda }{{{m_V}}} - \varepsilon {e^{ - ({\lambda _1} + \mu )(t - \tau )}}{\phi _2}){\phi _1} + (\frac{{N\dot \rho (t)}}{{\rho (t)}} + {m_V}){\phi _2})\\
 = \varepsilon {e^{ - ({\lambda _1} + \mu )(t - \tau )}}( - \mu {\phi _2} + \varepsilon {e^{ - ({\lambda _1} + \mu )(t - \tau )}}\frac{{{\beta _V}b}}{{{N_H} + L}}{\phi _2}{\phi _1})\\
 \le 0
\end{array}$$
for $t\in (0^+,\tau]$. It is clear that $\underline U (y,{0^ + }) = \underline U (y,0)$, $(\underline U (y,0),\underline V (y,0)) = (\underline U (y,\tau ),\underline V (y,\tau ))$, and $(\underline{U}(y,t), \underline{V}(y,t))=(0,0)$ for $y \in \partial {\Omega _0}$. Therefore, by Definition \ref{defi4.2}, $(\underline U (y,t),\underline V (y,t))$ is a lower solution to problem (\ref{a09}).

Using the monotone iterative method, we can prove the existence of $\tau$-periodic solutions to problem (\ref{a09}). Let
$${G_1}(t,U,V) = {M_1}U+{g_1}(t,U,V),\;\;{G_2}(t,U,V) = {M_2}V+ {g_2}(t,U,V),$$
where
$${M_1} = \frac{{{\beta _H}b}}{{{N_H} + L}}\overline C + {m_H} + {r_H} + \mathop {\sup }\limits_{t \in [0,\tau ]} \frac{{N\dot \rho (t)}}{{\rho (t)}},\;\;{M_2} = \frac{{{\beta _V}b}}{{{N_H} + L}}\overline C + {m_V} + \mathop {\sup }\limits_{t \in [0,\tau ]} \frac{{N\dot \rho (t)}}{{\rho (t)}}$$
with $\overline{C}$ defined in (\ref{d05}) and
$${g_1}(t,U,V) = \frac{{{\beta _H}b}}{{{N_H} + L}}(\frac{{{b_H}{N_H}}}{{{m_H}}} - U)V - (\frac{{N\dot \rho (t)}}{{\rho (t)}} + {m_H} + {r_H})U,$$
$${g_2}(t,U,V) = \frac{{{\beta _V}b}}{{{N_H} + L}}(\frac{\Lambda }{{{m_V}}} - V)U - (\frac{{N\dot \rho (t)}}{{\rho (t)}} + {m_V})V.$$
Then, $G_1(t,U,V)$ and $G_2(t,U,V)$ are nondecreasing with respect to $U$ and $V$, respectively. We select $(\overline{U}^{(0)}, \overline{V}^{(0)}) = (\overline{U}, \overline{V})$ and $(\underline{U}^{(0)}, \underline{V}^{(0)}) = (\underline{U}, \underline{V})$ as initial values and consider the iterative process
\begin{eqnarray}
\left\{ {\begin{array}{*{20}{l}}
{\frac{{\partial {{\overline U }^{(m)}}}}{{\partial t}} - \frac{{{d_H}}}{{{\rho ^2}(t)}}\Delta {{\overline U }^{(m)}} + {M_1}{{\overline U }^{(m)}} = {G_1}(t,{{\overline U }^{(m - 1)}},{{\overline V }^{(m - 1)}}),}&{y \in {\Omega _0},t \in ({0^ + },\tau ],}\\
{\frac{{\partial {{\underline U }^{(m)}}}}{{\partial t}} - \frac{{{d_H}}}{{{\rho ^2}(t)}}\Delta {{\underline U }^{(m)}} + {M_1}{{\underline U }^{(m)}} = {G_1}(t,{{\underline U }^{(m - 1)}},{{\underline V }^{(m - 1)}}),}&{y \in {\Omega _0},t \in ({0^ + },\tau ],}\\
{\frac{{\partial {{\overline V }^{(m)}}}}{{\partial t}} - \frac{{{d_V}}}{{{\rho ^2}(t)}}\Delta {{\overline V }^{(m)}} + {M_2}{{\overline V }^{(m)}} = {G_2}(t,{{\overline U }^{(m - 1)}},{{\overline V }^{(m - 1)}}),}&{y \in {\Omega _0},t \in ({0^ + },\tau ],}\\
{\frac{{\partial {{\underline V }^{(m)}}}}{{\partial t}} - \frac{{{d_V}}}{{{\rho ^2}(t)}}\Delta {{\underline V }^{(m)}} + {M_2}{{\underline V }^{(m)}} = {G_2}(t,{{\underline U }^{(m - 1)}},{{\underline V }^{(m - 1)}}),}&{y \in {\Omega _0},t \in ({0^ + },\tau ],}\\
{{{\overline U }^{(m)}}(y,{0^ + }) = {{\overline U }^{(m - 1)}}(y,\tau ),\;{{\underline U }^{(m)}}(y,{0^ + }) = {{\underline U }^{(m - 1)}}(y,\tau ),}&{y \in {\Omega _0},}\\
{{{\overline V }^{(m)}}(y,{0^ + }) = P({{\overline V }^{(m - 1)}}(y,\tau )),\;{{\underline V }^{(m)}}(y,{0^ + }) = P({{\underline V }^{(m - 1)}}(y,\tau )),}&{y \in {\Omega _0},}\\
{{{\overline U }^{(m)}}(y,t) = {{\underline U }^{(m)}}(y,t) = {{\overline V }^{(m)}}(y,t) = {{\underline V }^{(m)}}(y,t) = 0,}&{y \in \partial {\Omega _0},t \in [0,\tau ]}
\end{array}} \right.
\label{d06}
\end{eqnarray}
with periodic conditions
\begin{eqnarray}
\left\{ {\begin{array}{*{20}{l}}
{{{\overline U }^{(m)}}(y,0) = {{\overline U }^{(m - 1)}}(y,\tau ),\;{{\underline U }^{(m)}}(y,0) = {{\underline U }^{(m - 1)}}(y,\tau ),}&{y \in {{\overline \Omega  }_0},}\\
{{{\overline V }^{(m)}}(y,0) = {{\overline V }^{(m - 1)}}(y,\tau ),\;{{\underline V }^{(m)}}(y,0) = {{\underline V }^{(m - 1)}}(y,\tau ),}&{y \in {{\overline \Omega  }_0},m = 1,2, \cdots .}
\end{array}} \right.
\label{d07}
\end{eqnarray}
Accordingly, two sequences $\{ ({{\overline U}^{(m)}},{{\overline V}^{(m)}})\} $ and $\{ ({\underline U ^{(m)}},{\underline V ^{(m)}})\} $ are derived from (\ref{d06}) and (\ref{d07}), which, by Lemma 3.1 in \cite{pcv2005}, satisfy
$$({\underline U ^{(0)}},{\underline V ^{(0)}}) \le ({\underline U ^{(m - 1)}},{\underline V ^{(m - 1)}}) \le ({\underline U ^{(m)}},{\underline V ^{(m)}}) \le ({\overline U ^{(m)}},{\overline V ^{(m)}}) \le ({\overline U ^{(m - 1)}},{\overline V ^{(m - 1)}}) \le ({\overline U ^{(0)}},{\overline V ^{(0)}})$$
for any $(y,t) \in {\overline \Omega _0} \times [0,\infty) $. The monotone bounded convergence theorem shows that the existence of limits for $\{ ({{\overline U}^{(m)}},{{\overline V}^{(m)}})\} $ and $\{ ({\underline U ^{(m)}},{\underline V ^{(m)}})\} $, that is,
\begin{eqnarray}
\mathop {\lim }\limits_{t \to \infty } ({{\overline U}^{(m)}},{{\overline V}^{(m)}}) = ({U^*},{V^*}),\;\;\mathop {\lim }\limits_{t \to \infty } ({\underline U ^{(m)}},{\underline V ^{(m)}}) = ({U_*},{V_*}),
\label{d07'}
\end{eqnarray}
where $(U^*, V^*) $ and $(U_*, V_*)$ are $\tau$-periodic solutions to problem (\ref{a09}).  We claim that $(U^*, V^*) $ is the maximal positive $\tau$-periodic solution to problem (\ref{a09}), and $(U_*, V_*)$ is the minimal one. In fact, it is sufficient to choose $({{\overline U}^{(0)}},{{\overline V}^{(0)}}) = \left( {\frac{{{b_H}{N_H}}}{{{m_H}}},\frac{\Lambda }{{{m_V}}}} \right)$ and $({\underline U ^{(0)}},{\underline V ^{(0)}}) = (U,V)$ as the initial iteration values, where $(U, V)$ is any solution to problem (\ref{a09}) satisfying $(\underline U ,\underline V ) \le (U,V) \le (\overline U,\overline V)$, to derive $(U, V) \le (U^*, V^*)$ for any $(y,t) \in {\overline \Omega _0} \times [0,\infty) $. In a similar manner, $(U_*, V_*)$ is the minimal positive $\tau$-periodic solution. For more details, see the proof of Theorem 4.2 in \cite{zq2025}.

\noindent\textbf{Step 2.} Define
$$\Theta  = \{ \gamma  \in [0,1]\;|\;\gamma ({U_1},{V_1}) \le ({U_2},{V_2}),\;(y,t) \in {{\overline \Omega }_0} \times [0,\tau ]\} ,$$
where $({U_1},{V_1})$ and $({U_2},{V_2})$ are two distinct solutions of problem (\ref{a09}). If $1 \in \Theta$, then we derive the conclusion that $({U_1},{V_1}) \le ({U_2},{V_2})$ for $(y,t) \in {{\overline \Omega }_0} \times [0,\tau ]$. Similarly, we have $({U_2},{V_2}) \le ({U_1},{V_1})$ for $(y,t) \in {{\overline \Omega }_0} \times [0,\tau ]$. Thus, the solution to problem (\ref{a09}) is unique, that is, $({U^*},{V^*})=({U_*},{V_*})\triangleq(U, V)$.

Now, we prove $1 \in \Theta$ by contradiction. We assume that ${\gamma ^*} = \sup \Theta  < 1$. Calculations yield that
\begin{eqnarray}
\begin{array}{l}
\frac{{\partial ({V_2} - {\gamma ^*}{V_1})}}{{\partial t}} - \frac{{{d_V}}}{{{\rho ^2}(t)}}\Delta ({V_2} - {\gamma ^*}{V_1}) + {M_2}({V_2} - {\gamma ^*}{V_1})\\
 = {g_2}(t,U,{V_2}) + {M_2}{V_2} - {\gamma ^*}{g_2}(t,U,{V_1}) - {\gamma ^*}{M_2}{V_1}\\
 \ge {g_2}(t,U,{\gamma ^*}{V_1}) - {\gamma ^*}{g_2}(t,U,{V_1})\\
 > 0
\end{array}
\label{d08}
\end{eqnarray}
because $g_2(t,U,V) + {M_2}V$ is nondecreasing and ${g_2}(t,U,V)/V$ is strictly decreasing with respect to $V$, respectively. In light of Assumptions $\mathbf{(A1)}$ and $\mathbf{(A2)}$,
$${V_2}(y,{0^ + }) - {\gamma ^*}{V_1}(y,{0^ + }) = P({V_2}(y,0)) - {\gamma ^*}P({V_1}(y,0)) \ge P({\gamma ^*}{V_1}(y,0)) - {\gamma ^*}P({V_1}(y,0)) \ge 0$$
for $y\in\Omega_0$. Furthermore, the boundary condition $V_2(y,t) - \gamma^* V_1(y,t) = 0$ holds for $(y,t) \in \partial\Omega_0 \times [0,\tau]$. By the strong maximum principle\cite{Protter1984}, exactly one of the two relations $V_2 - \gamma^* V_1 > 0$ and $V_2 - \gamma^* V_1 \equiv 0$ must hold for $(y,t) \in {\Omega _0} \times [{0^ + } \cup ({0^ + },\tau ]]$. Analogously, $U_2 - \gamma^* U_1 > 0$ or $U_2 - \gamma^* U_1 \equiv 0$ for $(y,t) \in {\Omega _0} \times [{0^ + } \cup ({0^ + },\tau ]]$.

Suppose that $V_2 - \gamma^* V_1 > 0$ holds for $(y,t) \in \Omega_0 \times \left( \{0^+\} \cup (0^+, \tau] \right)$. In view of the periodicity conditions $V_1(y,0) = V_1(y,\tau)$ and $V_2(y,0) = V_2(y,\tau)$, we conclude that $V_2 - \gamma^* V_1 > 0 $ holds for $(y,t) \in \Omega_0 \times [0,\tau]$. The Hopf boundary lemma states that the outward normal derivative $\frac{\partial (V_2 - {\gamma ^*}{V_1})} {\partial \eta } < 0$ holds for $(y,t) \in \partial\Omega_0 \times [0,\tau]$. Consequently, there exists a constant  $\varepsilon_0>0$ such that $V_2 - \gamma^* V_1 > \varepsilon_0 V_1$, which contradicts $\gamma^* = \sup\Theta$.

Assume that $V_2 - \gamma^* V_1 \equiv 0$ holds for $(y,t) \in \Omega_0 \times \left( \{0^+\} \cup (0^+, \tau] \right)$. From the periodicity of $u$ and $v$, it follows that $V_2 - \gamma^* V_1 \equiv 0$ for $ (y,t) \in \Omega_0 \times [0,\tau]$. By substituting this result into the second equation of problem (\ref{a09}), we obtain $g_2(t,U,V_2) = g_2(t,U,\gamma^* V_1) = \gamma^* g_2(t,U,V_1)$, which contradicts (\ref{d08}).

As a result, the $\tau$-periodic solution $(U(y,t), V(y,t))$ to problem (\ref{a09}) exists and is unique. The proof of Theorem \ref{thm1.2} is finished. \epf\vspace{0.5\baselineskip}

Based on the above results, we present the convergence of the solution to the problem (\ref{a08}).

\vspace{0.5\baselineskip}\noindent\textbf{Proof of Theorem \ref{thm1.3}:} Without loss of generality, let $(u_0(y), v_0(y)) > (0,0)$ for $y \in \overline\Omega_0$. If not, we replace the initial time $0$ by some $t^*>0$. Via the Hopf boundary lemma, there exists a sufficiently small constant $\varepsilon$ for which $(u_0(y), v_0(y)) \ge (\varepsilon \phi_1(y,0), \varepsilon \phi_2(y,0)) $, where $(\phi_1,\phi_2)$ is defined in (\ref{c01}). Provided that $\overline{C}$ is a sufficiently large constant, the relation $ (\overline{C}, \overline{C}) \ge (u_0(y), v_0(y))$ is valid. From the proof of Theorem \ref{thm1.2}, it follows that
$$\begin{array}{l}
({\underline U ^{(0)}}(y,0),{\underline V ^{(0)}}(y,0)) = (\underline U (y,0),\underline V (y,0)) \le ({u_0}(y),{v_0}(y))\\
\quad\quad\quad\quad\quad\quad\quad\quad\quad\quad\quad\quad\quad\quad\quad\quad\;\,  \le (\overline U (y,0),\overline V (y,0)) = ({\overline U ^{(0)}}(y,0),{\overline V ^{(0)}}(y,0))
\end{array}$$
for $y \in \overline\Omega_0$. Given that $P(v)$ is monotone nondecreasing (Assumption $\mathbf{(A1)}$), $P(\underline{V}^{(0)}(y,0)) \le P(v_0(y)) \le P(\overline{V}^{(0)}(y,0))$, whence $$({\underline U ^{(0)}}(y,{0^ + }),{\underline V ^{(0)}}(y,{0^ + })) \le (u(y,{0^ + }),v(y,{0^ + })) \le ({\overline U ^{(0)}}(y,{0^ + }),{\overline V ^{(0)}}(y,{0^ + }))$$
for $y \in \overline\Omega_0$. In accordance with the comparison principle (Lemma \ref{lem4.1}), we obtain
$$(\underline U^{(0)} (y,t),\underline V^{(0)} (y,t)) \le (u(y,t),v(y,t)) \le (\overline U^{(0)}(y,t),\overline V^{(0)}(y,t))$$
for $(y,t) \in {{\overline \Omega }_0} \times [0,\tau ]$. Mathematical induction can extend the range of $t$ from $[0,\tau]$ to $[0,\infty)$.

Relation (\ref{d07}) shows that $$(\underline{U}^{(0)}(y,\tau), \underline{V}^{(0)}(y,\tau)) = (\underline{U}^{(1)}(y,0), \underline{V}^{(1)}(y,0)),\;\;(\overline{U}^{(0)}(y,\tau), \overline{V}^{(0)}(y,\tau)) = (\overline{U}^{(1)}(y,0), \overline{V}^{(1)}(y,0)).$$ For $t = \tau$, we have
$$(\underline{U}^{(1)}(y,0), \underline{V}^{(1)}(y,0)) \le (u(y,\tau), v(y,\tau)) \le (\overline{U}^{(1)}(y,0), \overline{V}^{(1)}(y,0))$$
for $y \in \overline\Omega_0$. Upon repeating the above steps, one arrives at
$$({\underline U ^{(1)}}(y,t),{\underline V ^{(1)}}(y,t)) \le (u(y,t + \tau ),v(y,t + \tau )) \le ({{\overline U}^{(1)}}(y,t),{{\overline V}^{(1)}}(y,t))$$
for $(y,t) \in {{\overline \Omega }_0} \times [0,\infty)$. Step by step, we have
$$({\underline U ^{(n)}}(y,t),{\underline V ^{(n)}}(y,t)) \le (u(y,t + n\tau ),v(y,t + n\tau )) \le ({{\overline U}^{(n)}}(y,t),{{\overline V}^{(n)}}(y,t)),\;\; n=0,1,2,\cdots$$
for $(y,t) \in {{\overline \Omega }_0} \times [0,\infty)$. From Theorem \ref{thm1.2} and (\ref{d07'}), we see that
$$\lim_{n \to \infty} (\underline{U}^{(n)}(y,t), \underline{V}^{(n)}(y,t)) = \lim_{n \to \infty} (\overline{U}^{(n)}(y,t), \overline{V}^{(n)}(y,t)) = (U(y,t),V(y,t))$$
for $t \in [0,\infty )$. Hence, relation (\ref{a10}) holds, and Theorem \ref{thm1.3} is proven.\epf\vspace{0.5\baselineskip}

In conjunction with Definition \ref{defi4.1} and Lemma \ref{lem4.1}, we proceed to prove Theorem \ref{thm1.1}.\vspace{0.5\baselineskip}

\noindent\textbf{Proof of Theorem \ref{thm1.1}:} We first consider the case where $\lambda_1 > 0$. If we can find an upper solution $(\overline u (y,t), \overline v (y,t))$ to problem (\ref{a08}) such that
\begin{eqnarray}
\mathop {\lim }\limits_{t \to \infty } (\overline u (y,t),\overline v (y,t)) = (0,0),\quad y \in {{\overline \Omega }_0}.
\label{d04}
\end{eqnarray}
Then by Lemma \ref{lem4.1},
$$(u(y,t), v(y,t)) \le (\overline u (y,t), \overline v (y,t)),\quad y \in {{\overline \Omega }_0},t \in [0,\infty )$$
holds for any solution $(u(y,t), v(y,t))$ to problem (\ref{a08}), thereby proving Theorem \ref{thm1.1} for $\lambda_1 > 0$. Now we construct such an upper solution. Let
$$\overline u(y,t) = M e^{-\lambda_1 t} \phi_1(y,t),\;\;\overline v(y,t) = M e^{-\lambda_1 t} \phi_2(y,t),$$
where $M$ is a sufficiently large constant such that $M \phi_1(y,0) \ge u_0(y)$ and $M \phi_2(y,0) \ge v_0(y)$ for any given initial value $( u_0(y), v_0(y))$, as well as $\phi_1(y,t)$ and $\phi_2(y,t)$ are positive eigenfunctions of problem (\ref{c01}). Direct computation gives
$$\begin{array}{l}
\frac{{\partial \overline u}}{{\partial t}} - \frac{{{d_H}}}{{{\rho ^2}(t)}}\Delta \overline u - \frac{{{\beta _H}b}}{{{N_H} + L}}(\frac{{{b_H}{N_H}}}{{{m_H}}} - \overline u)\overline v + (\frac{{N\dot \rho (t)}}{{\rho (t)}} + {m_H} + {r_H})\overline u\\
 \ge \frac{{\partial \overline u}}{{\partial t}} - \frac{{{d_H}}}{{{\rho ^2}(t)}}\Delta \overline u - \frac{{{\beta _H}b}}{{{N_H} + L}}\frac{{{b_H}{N_H}}}{{{m_H}}}\overline v + (\frac{{N\dot \rho (t)}}{{\rho (t)}} + {m_H} + {r_H})\overline u\\
 = M{e^{ - {\lambda _1}t}}( - {\lambda _1}{\phi _1} + \frac{{\partial {\phi _1}}}{{\partial t}} - \frac{{{d_H}}}{{{\rho ^2}(t)}}\Delta {\phi _1} - \frac{{{\beta _H}b}}{{{N_H} + L}}\frac{{{b_H}{N_H}}}{{{m_H}}}{\phi _2} + (\frac{{N\dot \rho (t)}}{{\rho (t)}} + {m_H} + {r_H}){\phi _1})\\
 = 0
\end{array}$$
and
$$\begin{array}{l}
\frac{{\partial \overline v}}{{\partial t}} - \frac{{{d_V}}}{{{\rho ^2}(t)}}\Delta \overline v - \frac{{{\beta _V}b}}{{{N_H} + L}}(\frac{\Lambda }{{{m_V}}} - \overline v)\overline u + (\frac{{N\dot \rho (t)}}{{\rho (t)}} + {m_V})\overline v\\
 \ge \frac{{\partial \overline v}}{{\partial t}} - \frac{{{d_V}}}{{{\rho ^2}(t)}}\Delta \overline v - \frac{{{\beta _V}b}}{{{N_H} + L}}\frac{\Lambda }{{{m_V}}}\overline u + (\frac{{N\dot \rho (t)}}{{\rho (t)}} + {m_V})\overline v\\
 = M{e^{ - {\lambda _1}t}}( - {\lambda _1}{\phi _2} + \frac{{\partial {\phi _2}}}{{\partial t}} - \frac{{{d_V}}}{{{\rho ^2}(t)}}\Delta {\phi _2} - \frac{{{\beta _V}b}}{{{N_H} + L}}\frac{\Lambda }{{{m_V}}}{\phi _1} + (\frac{{N\dot \rho (t)}}{{\rho (t)}} + {m_V}){\phi _2})\\
 = 0.
\end{array}$$
In addition, based on Assumptions $\mathbf{(A1)}$ and $\mathbf{(A2)}$, we obtain
$$\overline u(y,{(n\tau )^ + }) - \overline u(y,n\tau ) = M{e^{ - {\lambda _1}n\tau }}({\phi _1}(y,{(n\tau )^ + }) - {\phi _1}(y,n\tau )) = 0$$
and
$$\overline v(y,{(n\tau )^ + }) - P(\overline v(y,n\tau )) = M{e^{ - {\lambda _1}n\tau }}{\phi _2}(y,{(n\tau )^ + } - P(\overline v(y,n\tau )) = P'(0)\overline v(y,n\tau ) - P(\overline v(y,n\tau )) \ge 0.$$
Therefore, in accordance with Definition \ref{defi4.1}, $(\overline u(y,t), \overline v(y,t))$ is an upper solution corresponding to problem (\ref{a08}), and (\ref{d04}) is satisfied.

For $\lambda_1 = 0$, we consider the following auxiliary problem:
\begin{eqnarray}
\left\{ {\begin{array}{*{20}{l}}
{\frac{{\partial U}}{{\partial t}} = \frac{{{d_H}}}{{{\rho ^2}(t)}}\Delta U + \frac{{{\beta _H}b}}{{{N_H} + L}}(\frac{{{b_H}{N_H}}}{{{m_H}}} - U)V - (\frac{{N\dot \rho (t)}}{{\rho (t)}} + {m_H} + {r_H})U + \lambda U,}&{y \in {\Omega _0},t \in ({0^ + },\tau ],}\\
{\frac{{\partial V}}{{\partial t}} = \frac{{{d_V}}}{{{\rho ^2}(t)}}\Delta V + \frac{{{\beta _V}b}}{{{N_H} + L}}(\frac{\Lambda }{{{m_V}}} - V)U - (\frac{{N\dot \rho (t)}}{{\rho (t)}} + {m_V})V + \lambda V,}&{y \in {\Omega _0},t \in ({0^ + },\tau ],}\\
{U(y,{0^ + }) = U(y,0),\;V(y,{0^ + }) = P(V(y,0)),}&{y \in {\Omega _0},}\\
{U(y,t) = V(y,t) = 0,}&{y \in \partial {\Omega _0},t \in (0,\tau ],}\\
{U(y,0) = U(y,\tau ),\;V(y,0) = V(y,\tau ),}&{y \in {{\bar \Omega }_0},}
\end{array}} \right.
\label{d09}
\end{eqnarray}
where $\lambda$ is an arbitrary constant. It is easily seen that when $\lambda\le\lambda_1$, problem (\ref{d09}) only admits the zero solution. When $\lambda > \lambda_1$, it follows from Theorem \ref{thm1.2} that problem (\ref{d09}) has a unique positive $\tau$-periodic solution, denoted $(U_{\lambda}(y,t),V_{\lambda}(y,t))$. Clearly, $(U_{\lambda}(y,t),V_{\lambda}(y,t))$ is monotone increasing with respect to
$\lambda$ and satisfies
\begin{eqnarray}
\mathop {\lim }\limits_{\lambda  \to {\lambda _1}} {\left\| {({U_\lambda },{V_\lambda })} \right\|_{C({{\overline \Omega  }_0} \times [0,\tau ])}} = (0,0).
\label{d10}
\end{eqnarray}

Let $(u_{\lambda}(y,t),v_{\lambda}(y,t))$ be the solution of the following problem
\begin{eqnarray*}
\left\{ {\begin{array}{*{20}{l}}
{\frac{{\partial u}}{{\partial t}} = \frac{{{d_H}}}{{{\rho ^2}(t)}}\Delta u + \frac{{{\beta _H}b}}{{{N_H} + L}}(\frac{{{b_H}{N_H}}}{{{m_H}}} - u)v - (\frac{{N\dot \rho (t)}}{{\rho (t)}} + {m_H} + {r_H})u + \lambda,}&{y \in {\Omega _0},t \in ({{(n\tau )}^ + },(n + 1)\tau ],}\\
{\frac{{\partial v}}{{\partial t}} = \frac{{{d_V}}}{{{\rho ^2}(t)}}\Delta v + \frac{{{\beta _V}b}}{{{N_H} + L}}(\frac{\Lambda }{{{m_V}}} - v)u - (\frac{{N\dot \rho (t)}}{{\rho (t)}} + {m_V})v + \lambda,}&{y \in {\Omega _0},t \in ({{(n\tau )}^ + },(n + 1)\tau ],}\\
{u(y,{{(n\tau )}^ + }) = u(y,n\tau ),\;v(y,{{(n\tau )}^ + }) = P(v(y,n\tau )),}&{y \in {\Omega _0},}\\
{u(y,t) = v(y,t) = 0,}&{y \in \partial {\Omega _0},t > 0,}\\
{u(y,0) = {I_{H,0}}(x(0)) \le \frac{{{b_H}{N_H}}}{{{m_H}}},\;v(y,0) = {I_{V,0}}(x(0)) \le \frac{\Lambda }{{{m_V}}},}&{y \in {{\overline \Omega }_0},n = 0,1, \cdots .}
\end{array}} \right.
\end{eqnarray*}
We thus find that $(u_{\lambda}(y,t),v_{\lambda}(y,t))$ fulfills
$$(u(y,t),v(y,t)) \le ({u_\lambda }(y,t),{v_\lambda }(y,t)),\quad y \in {{\bar \Omega }_0},t \in [0,\infty )$$
for $\lambda > \lambda_1$. By Theorem \ref{thm1.3} we further obtain
$$\mathop {\lim }\limits_{n \to \infty } \sup (u(y,n\tau  + t),v(y,n\tau  + t)) \le \mathop {\lim }\limits_{n \to \infty } ({u_\lambda }(y,n\tau  + t),{v_\lambda }(y,n\tau  + t)) = ({U_\lambda }(y,t),{V_\lambda }(y,t))$$
for $y \in {{\bar \Omega }_0},t \in [0,\infty )$. Taking $\lambda  \to {\lambda _1} \ge 0$ and combining with (\ref{d10}) yields
$$\mathop {\lim }\limits_{t \to \infty } \sup {\left\| {(u( \cdot ,t),v( \cdot ,t))} \right\|_{C({{\overline \Omega  }_0})}} = (0,0),$$
that is,
$$\mathop {\lim }\limits_{t \to \infty } {\left\| {(u( \cdot ,t),v( \cdot ,t))} \right\|_{C({{\overline \Omega  }_0})}} = (0,0).$$
This completes the proof of Theorem \ref{thm1.1}.\epf\vspace{0.5\baselineskip}

\section{Eigenvalue estimate, numerical simulation and discussion}

The existence and related properties of the principal eigenvalue ${\lambda _1}$ were established in Section 3. This section investigates how the initial region ${\Omega _0}$ and the impulsive intensity $P'(0)$ influence infectious disease dynamics.

\subsection{Eigenvalue estimate}
For the subsequent numerical simulations, an explicit expression for the threshold ${\lambda _1}$ is required. In problem (\ref{c01}), however, the existence of $P'(0)$ makes the calculation of the principal eigenvalue ${\lambda _1}$ complicated. Nevertheless, numerical simulation remains feasible by leveraging specific estimates and the known properties of ${\lambda _1}$.

In the following, we derive estimates and explicit expressions for the principal eigenvalue ${\lambda _1}$ in the special cases. For convenience, we define $\overline {{\rho ^{ - 2}}}  =
\frac{1}{ \tau }\int_0^\tau  {{\rho ^{ - 2}}(t)} dt$.
\begin{case}\label{case5.1}
Suppose $P'(0)=1$, then ${\lambda _1}$ satisfies
\begin{eqnarray*}
\begin{array}{l}
\displaystyle{\lambda _1} \le \left( {({d_H} + {d_V}){\lambda ^*}\overline {{\rho ^{ - 2}}}  + {m_H} + {r_H} + {m_V}} \right)/2\\
\displaystyle \quad\quad - {\left[ {{{\left( {({d_H} - {d_V}){\lambda ^*}\overline {{\rho ^{ - 2}}}  + {m_H} + {r_H} - {m_V}} \right)}^2} + 4\frac{{{\beta _H}{\beta _V}{b^2}}}{{{{({N_H} + L)}^2}}}\frac{{{b_H}{N_H}}}{{{m_H}}}\frac{\Lambda }{{{m_V}}}} \right]^{1/2}}/2.
\end{array}
\label{e01}
\end{eqnarray*}
\end{case}
\begin{case}\label{case5.2}
(i)\;Suppose $P'(0)=1$, ${{d_H} = {d_V}}$, and $\frac{{{\beta _H}b}}{{{N_H} + L}}\frac{{{b_H}{N_H}}}{{{m_H}}} = \frac{{{\beta _V}b}}{{{N_H} + L}}\frac{\Lambda }{{{m_V}}}:= c_0$, then ${\lambda _1}$ satisfies
\begin{eqnarray*}
{\lambda _1} \ge {d_H}{\lambda ^*}\overline {{\rho ^{ - 2}}}  - {c_0}.
\label{e06}
\end{eqnarray*}
(ii)\;If ${d_H} \ne {d_V}$ and $\frac{{{\beta _H}b}}{{{N_H} + L}}\frac{{{b_H}{N_H}}}{{{m_H}}} \ne \frac{{{\beta _V}b}}{{{N_H} + L}}\frac{\Lambda }{{{m_V}}}$, then
$${\lambda _1} \ge \min \{ {d_H},{d_V}\} {\lambda ^*}\overline {{\rho ^{ - 2}}}  - \max \{ \frac{{{\beta _H}b}}{{{N_H} + L}}\frac{{{b_H}{N_H}}}{{{m_H}}},\frac{{{\beta _V}b}}{{{N_H} + L}}\frac{\Lambda }{{{m_V}}}\} .$$
\end{case}
\begin{case}\label{case5.3}
Suppose $P'(0)=1$, ${{d_H} = {d_V}}$, and $\rho (t) \equiv 1$, then ${\lambda _1}$ satisfies
\begin{eqnarray*}
{\lambda _1} = {d_H}{\lambda ^*} + \left[ {{m_H} + {r_H} + {m_V} - {{\left( {{{\left( {{m_H} + {r_H} - {m_V}} \right)}^2} + 4\frac{{{\beta _H}{\beta _V}{b^2}}}{{{{({N_H} + L)}^2}}}\frac{{{b_H}{N_H}}}{{{m_H}}}\frac{\Lambda }{{{m_V}}}} \right)}^{1/2}}} \right]/2.
\label{e08}
\end{eqnarray*}
(ii)\;If ${d_H} \ne {d_V}$, then
$$\begin{array}{l}
\displaystyle{\lambda _1} = \left( {({d_H} + {d_V}){\lambda ^*} + {m_H} + {r_H} + {m_V}} \right)/2\\
\displaystyle \quad\quad - {\left[ {{{\left( {({d_H} - {d_V}){\lambda ^*} + {m_H} + {r_H} - {m_V}} \right)}^2} + 4\frac{{{\beta _H}{\beta _V}{b^2}}}{{{{({N_H} + L)}^2}}}\frac{{{b_H}{N_H}}}{{{m_H}}}\frac{\Lambda }{{{m_V}}}} \right]^{1/2}}/2.
\end{array}$$
\end{case}

To establish the statements above, it suffices to adapt Theorem 5.3 in \cite{zq2025} with only a few small changes.

\subsection{Numerical simulation}

The purpose of the numerical simulations in this study is twofold: to verify our theoretical results and to obtain a more direct understanding of the effects of different parameters on the transmission of dengue fever.

We now fix the following epidemiological parameters in (\ref{a08}):
$$N_H=35,\;\;\Lambda=42,\;\;m_H=0.1,\;\;b_H=0.15,\;\;r_H=0.3,$$\vspace{-0.8cm}
$$L=1,\;\;\beta_H=8,\;\;\beta_V=4,\;\;N=1,\;\;\tau =2.$$
Next, the initial region is chosen as ${\Omega _0} = [0,\pi ]$ with ${\lambda ^*} = 1$, and the initial functions are set to
$${u_0}(y) = 0.5\sin y + 0.2\sin 3y,\;\;{v_0}(y) = 2\sin y + \sin 3y.$$

Our simulations unfold along two paths. One is designed to explore the influence of periodic impulsive control on disease transmission, and the other to assess the effect of the scaling factor.

\begin{exm}\label{exm5.1} Fix $d_H=0.5,d_V=0.2,b=0.07$ and $m_V=0.55$. The impulsive function and the scaling factor are selected in three combinations: $P(v) = v,\rho (t) \equiv 1$; $P(v) =\frac{{7v}}{10+v},\rho (t) \equiv 1$; and $P(v) = v,\rho (t)={0.75e^{2(1 - \cos \pi t)}}$.
\end{exm}

\begin{figure}[htbp]
\centering
\subfigure[]{ {
\includegraphics[width=0.4\textwidth]{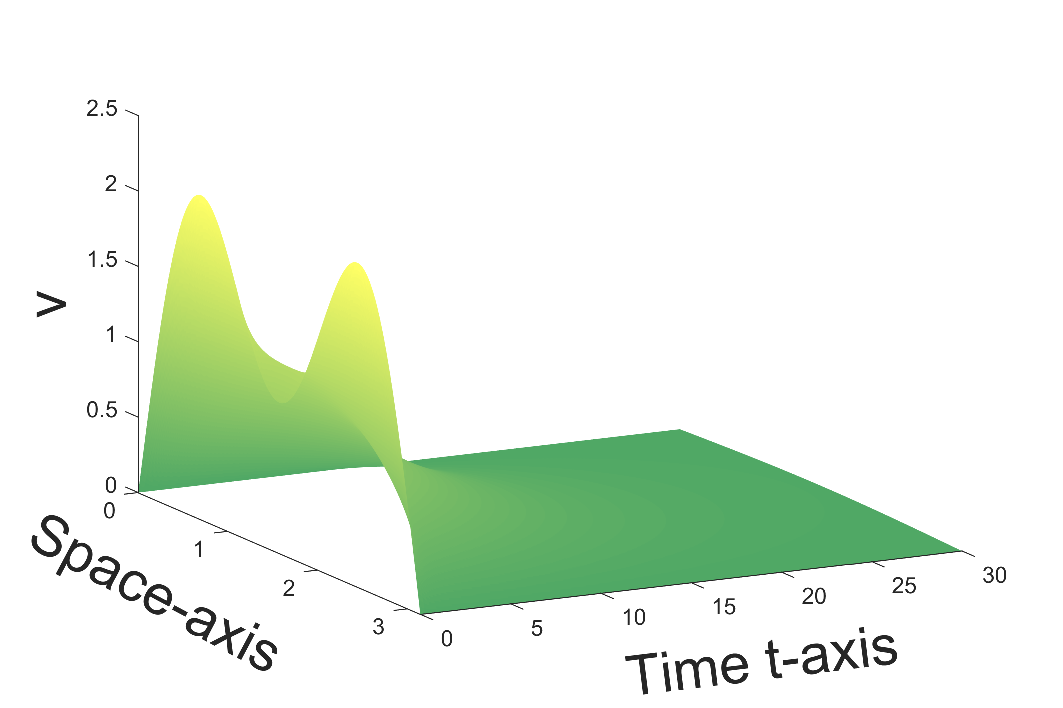}
} }
\subfigure[]{ {
\includegraphics[width=0.4\textwidth]{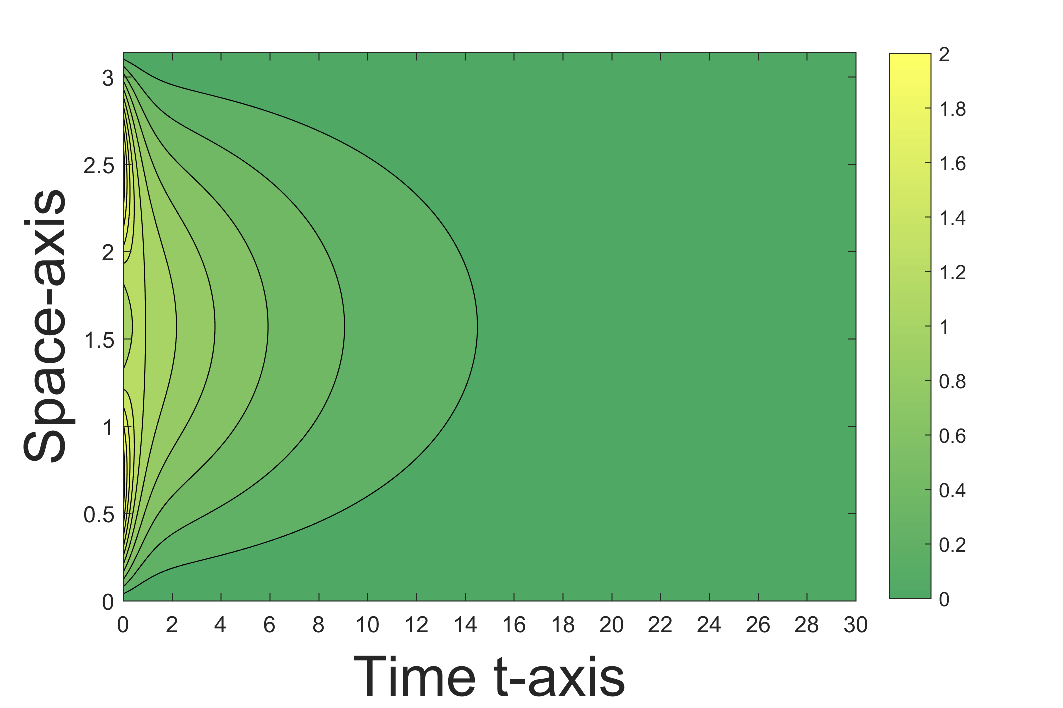}
} }
\subfigure[]{ {
\includegraphics[width=0.4\textwidth]{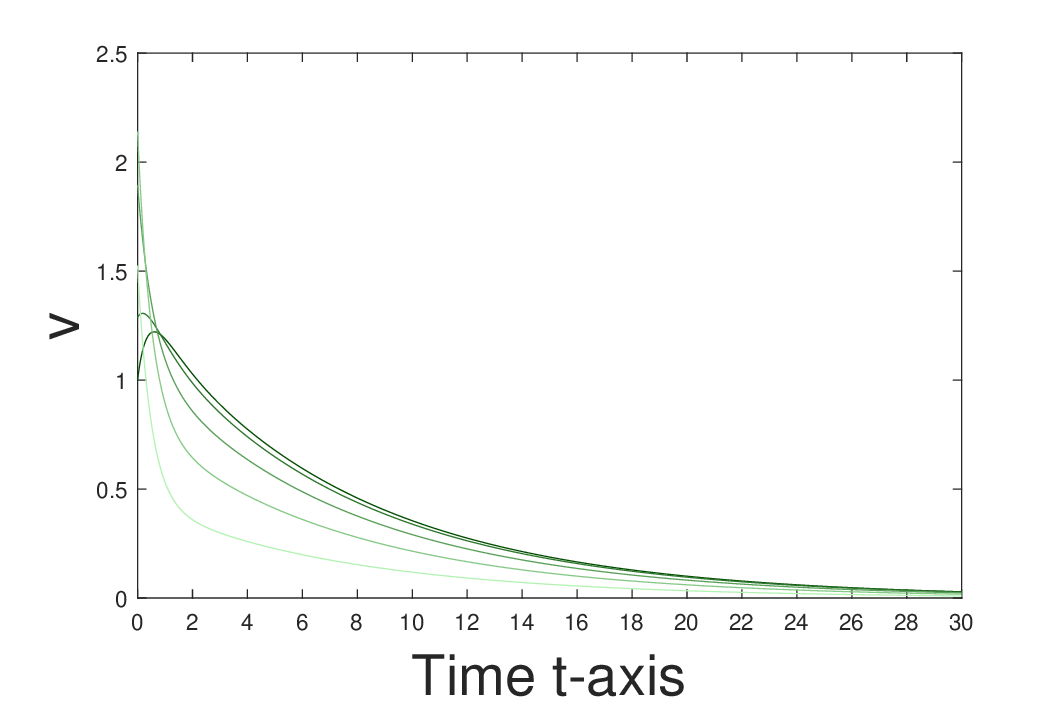}
} }
\renewcommand{\figurename}{\scriptsize Fig.}
\caption{\scriptsize $P(v) = v$ and $\rho (t) \equiv 1$. From graphs (a)-(f), it can be seen that the variables $u$ and $v$ ultimately approach zero. Graphs (a) and (d), (b) and (e), as well as (c) and (f) are a 3D surface plot, a contour plot, and a spatial projection plot, respectively.}
\end{figure}

In the case where $P(v) = v$ and $\rho (t) \equiv 1$ (so that there is no impulsive control and no domain evolution), the calculation in Case \ref{case5.3}(ii) yields ${\lambda _1} \approx 0.125 > 0$. Theorem \ref{thm1.1} states that when $\lambda_1 > 0$, the solution of problem (\ref{a08}) converges to the trivial solution, which matches the results shown in Fig.\,2(a-f). Moreover, Fig.\,2(c,f) illustrates that a lower mosquito biting rate and a higher mosquito mortality rate effectively control the spread of the dengue fever.

\begin{figure}[htbp]
\centering
\subfigure[]{ {
\includegraphics[width=0.4\textwidth]{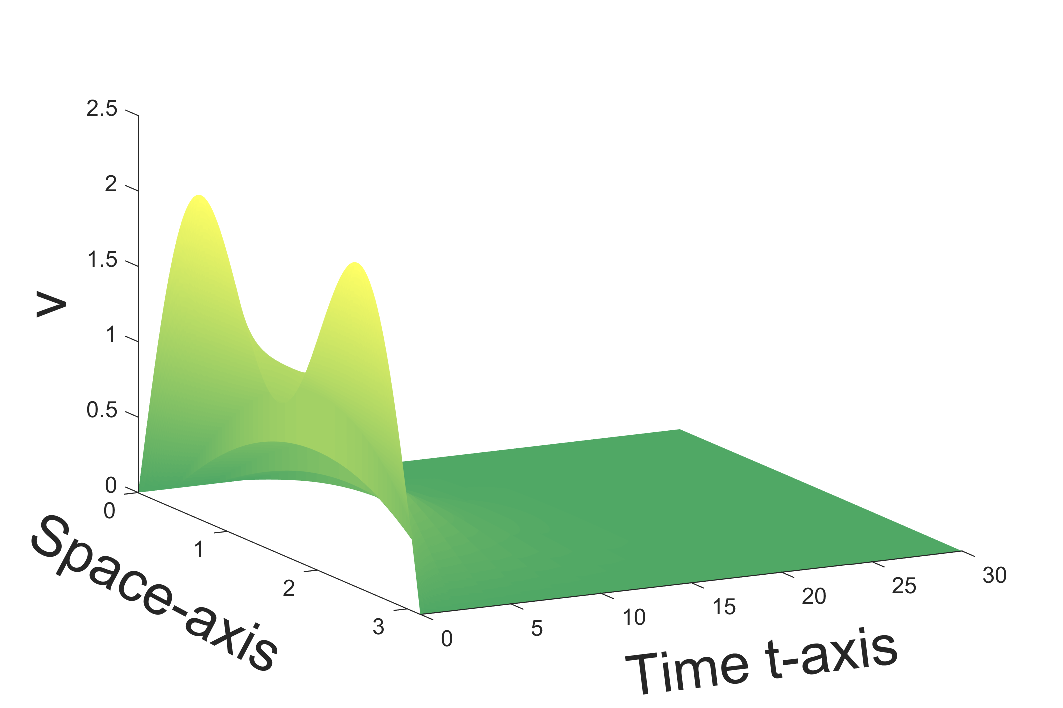}
} }
\subfigure[]{ {
\includegraphics[width=0.4\textwidth]{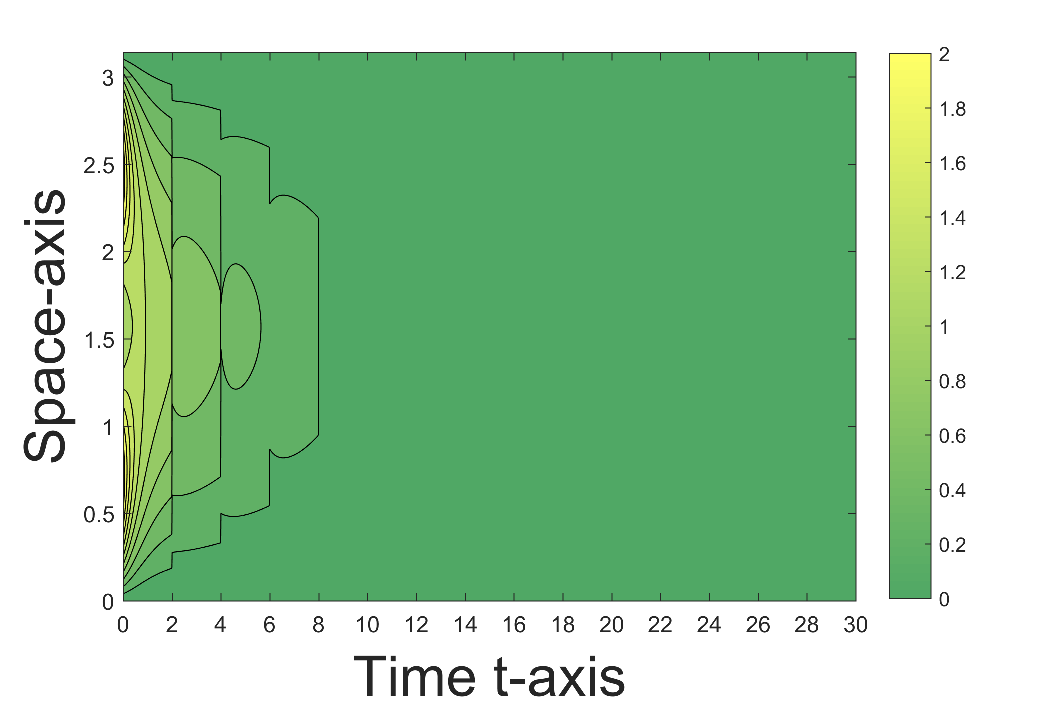}
} }
\subfigure[]{ {
\includegraphics[width=0.4\textwidth]{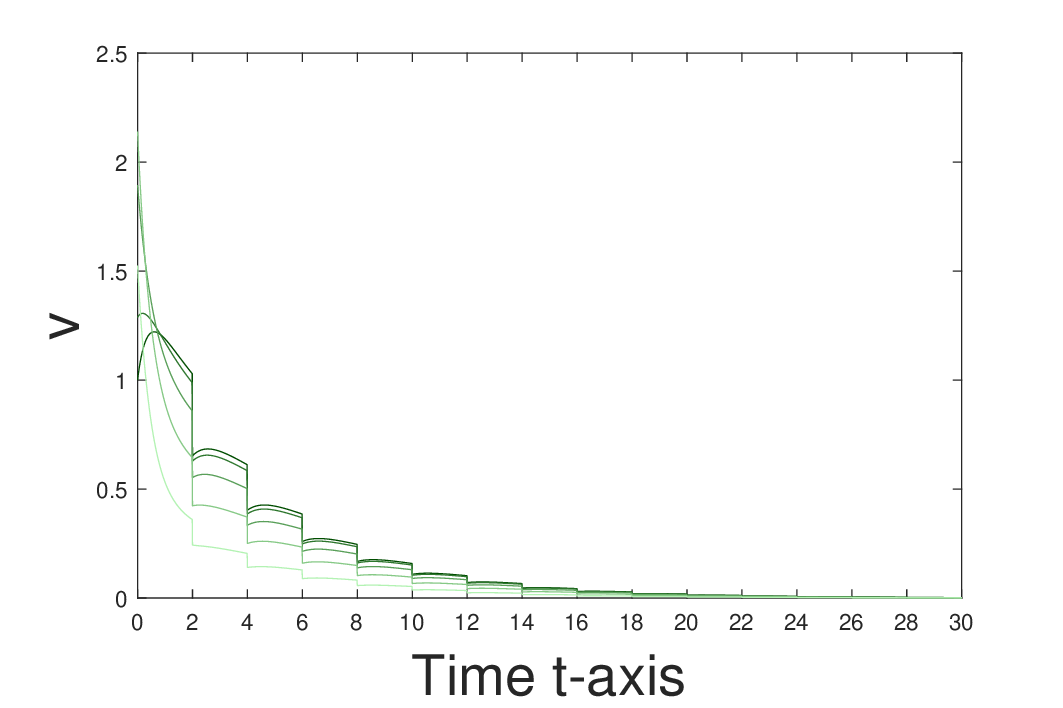}
} }
\renewcommand{\figurename}{\scriptsize Fig.}
\caption{\scriptsize $P(v) = \frac{{7v}}{{v + 10}}$ and $\rho (t) \equiv 1$. Graphs (a)-(f) indicate that the variable $v$ ( and $u$ omitted by simplification) eventually converge to zero. As shown in graph (f), the impulsive control is implemented at every time $\tau = 2,4, \cdots$.}
\end{figure}

For the case $P(v) = \frac{{7v}}{{v + 10}}$ and $\rho (t) \equiv 1$ (so that impulsive control is active but the domain remains fixed), calculation using Theorem \ref{thm3.4} and Case \ref{case5.3}(ii) gives ${\lambda _1}(0.7) > {\lambda _1}(1) \approx 0.125 > 0$. By Theorem \ref{thm1.1}, the solution of problem (1.6) also converges to zero, which is seen in Fig.\,3(a-f). A comparison of graphs (c) and (f) in Figs.\,2 and 3 shows that both variables $u$ and $v$ tend to zero earlier than before. This implies a faster decline in dengue fever incidence. Consequently, periodic impulsive control favors the elimination of this epidemic.
\begin{figure}[H]
\centering
\subfigure[]{ {
\includegraphics[width=0.4\textwidth]{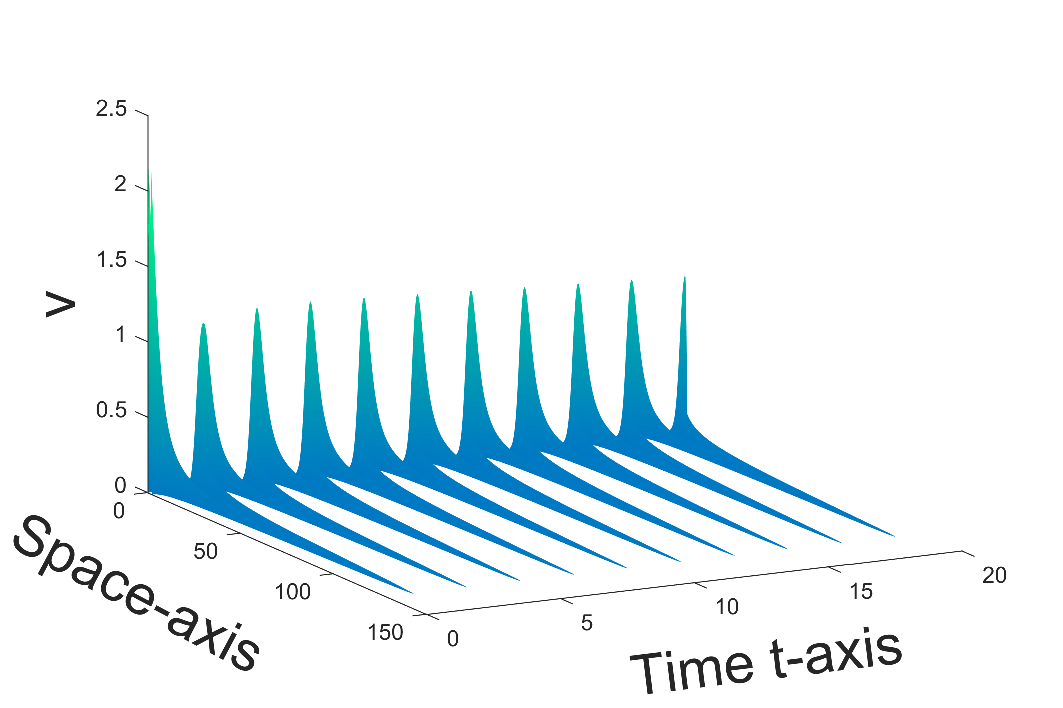}
} }
\subfigure[]{ {
\includegraphics[width=0.4\textwidth]{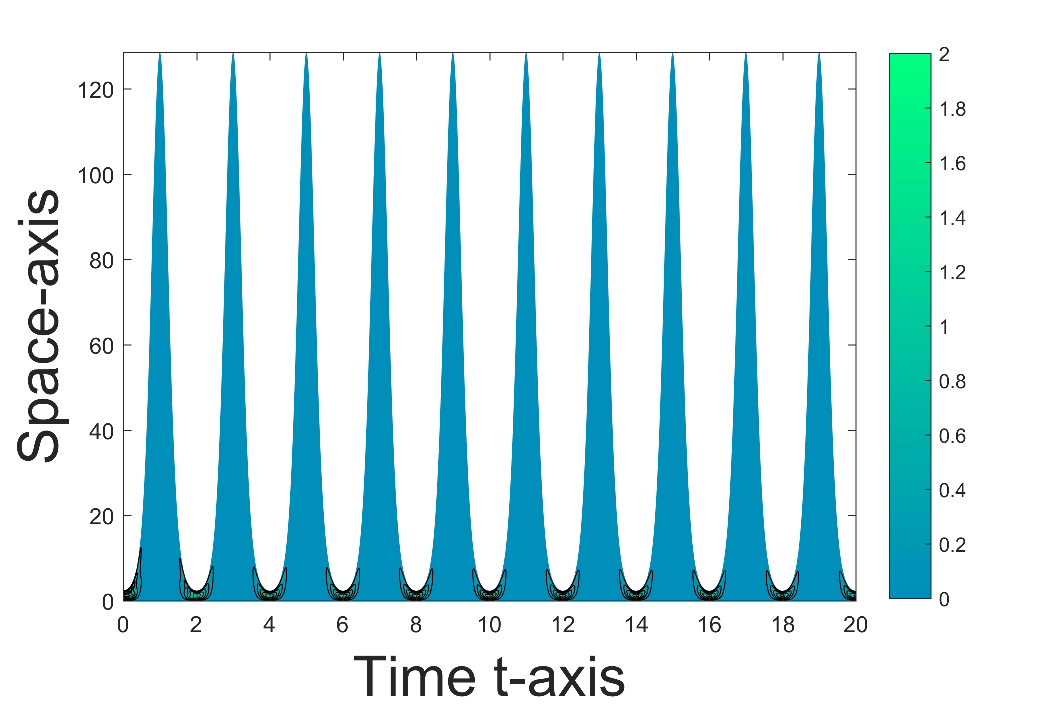}
} }
\subfigure[]{ {
\includegraphics[width=0.4\textwidth]{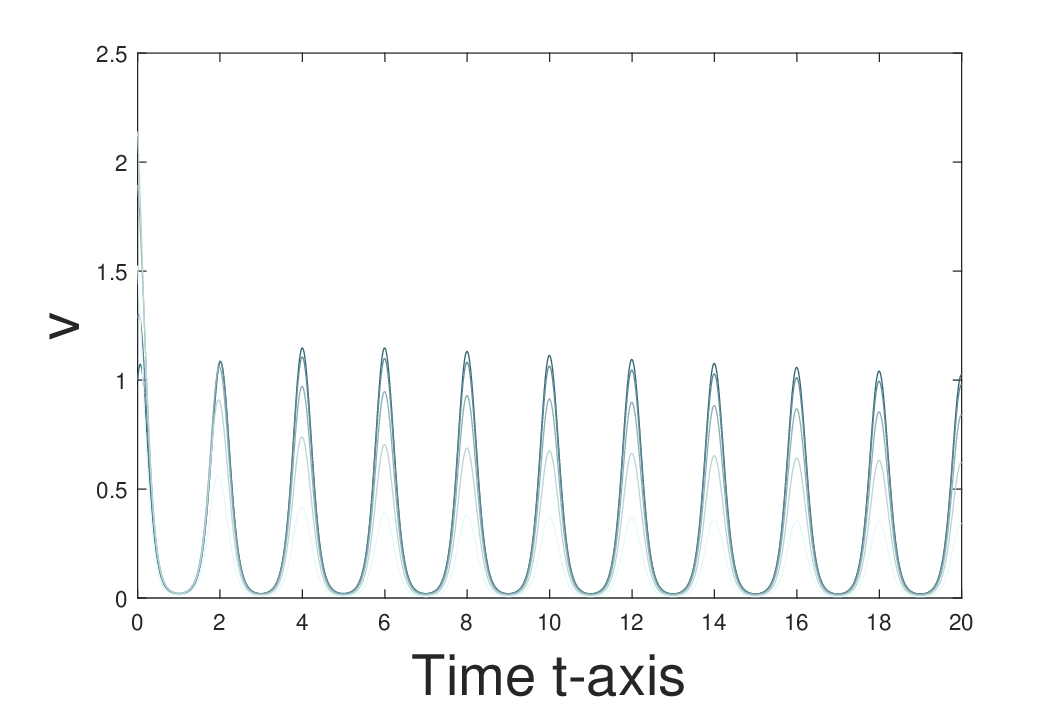}
} }
\renewcommand{\figurename}{\scriptsize Fig.}
\caption{\scriptsize $P(v) = v$ and $\rho (t)={0.75e^{2(1 - \cos \pi t)}}$. Graphs (a)-(f) show that the variable $v$  ( and $u$ omitted by simplification)  tend to a positive periodic steady state. Graphs (b) and (e) clearly illustrate the periodic variation of the domain ${\Omega _t}$ over time.}
\end{figure}

In this scenario where $P(v) = v$ and $\rho (t)={0.75e^{2(1 - \cos 0.5\pi t)}}$ (so that only domain evolution is present), calculation based on Case \ref{case5.1} yields ${\lambda _1} \le  - 0.093 < 0$. According to Theorem \ref{thm1.3}, the solution $(u,v)$ of problem (\ref{a08}) converges to a positive $\tau$-periodic solution of the steady state problem (\ref{a09}), which corresponds to the behavior shown in Fig.\,4(a-f). The contrast between Figs.\,2 and 4 reveals that enlarging the habitat region can turn an otherwise extinct infection into an endemic. In other words, a larger scaling factor hinders disease eradication.

Example \ref{exm5.1} corresponds to the scenario where impulsive control and domain variation are applied after the variables $u$ and $v$ have approached zero. We next investigate how these factors influence the subsequent dynamics once $u$ and $v$ have approached a positive equilibrium.

\begin{exm}\label{exm5.2} Fix $d_H=0.75,d_V=0.7,b=0.15$ and, $m_V=0.4$. The impulsive function and the scaling factor are selected in three combinations: $P(v) = v,\rho (t) \equiv 1$; $P(v) =\frac{{9.9v}}{10+v},\rho (t) \equiv 1$; and $P(v) = v,\rho (t)={0.75e^{-0.16(1 - \cos \pi t)}}$.
\end{exm}

\begin{figure}[htbp]
\centering
\subfigure[]{ {
\includegraphics[width=0.4\textwidth]{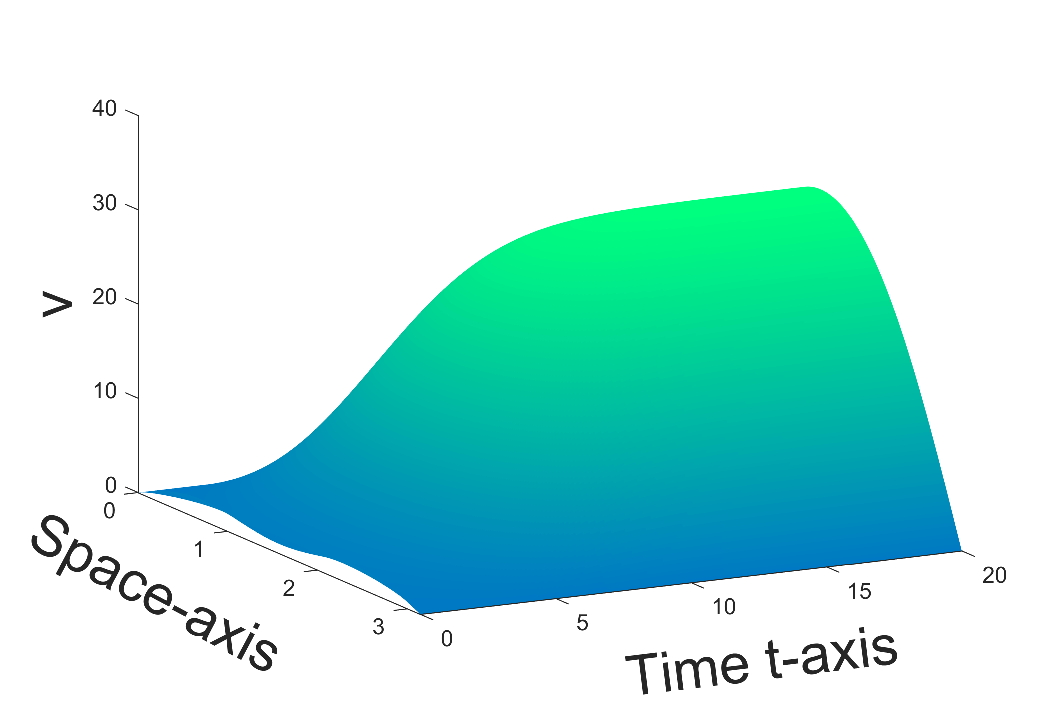}
} }
\subfigure[]{ {
\includegraphics[width=0.4\textwidth]{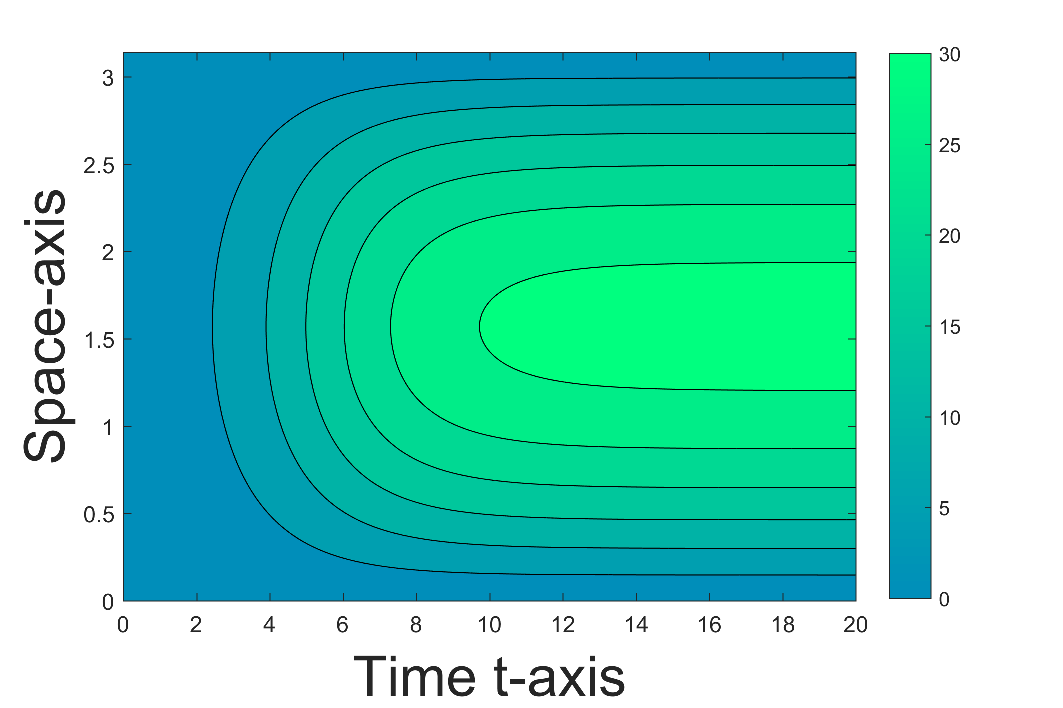}
} }
\subfigure[]{ {
\includegraphics[width=0.4\textwidth]{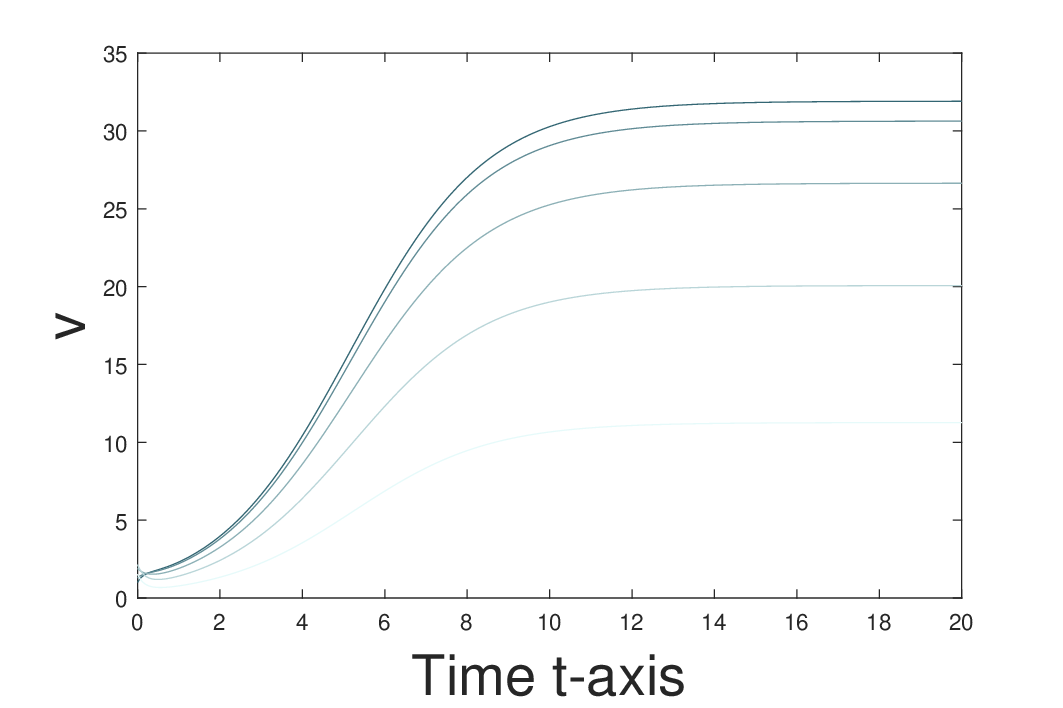}
} }
\renewcommand{\figurename}{\scriptsize Fig.}
\caption{\scriptsize $P(v) = v$ and $\rho (t) \equiv 1$. As shown in graphs (a)-(f), the variable $v$  ( and $u$ omitted by simplification) tend to a positive steady state.}
\end{figure}

Theorem \ref{thm1.3} shows that when $\lambda_1 =-0.625< 0$, the solution of problem (\ref{a08}) converges to a positive periodic steady state for any initial value, which corresponds to the results shown in Fig.\,5(a)-(f). A comparison between Figs.\,2 and 5 reveals that increasing the mosquito biting rate while decreasing its mortality rate can cause the dengue fever to transition from dying out to sustained transmission.

\begin{figure}[htbp]
\centering
\subfigure[]{ {
\includegraphics[width=0.4\textwidth]{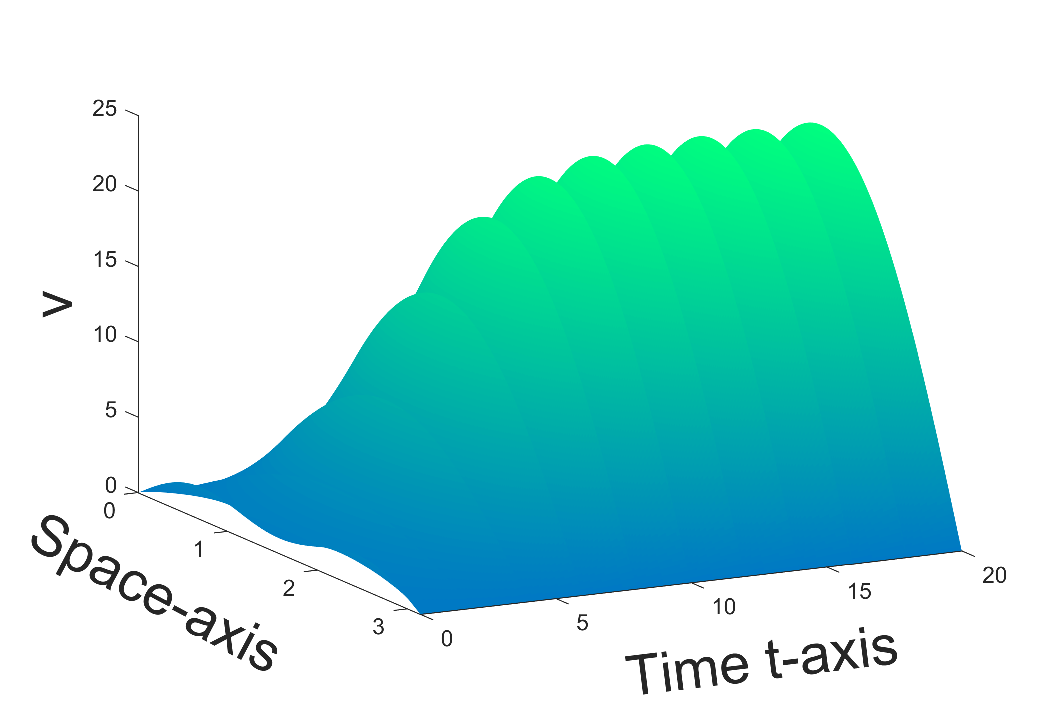}
} }
\subfigure[]{ {
\includegraphics[width=0.4\textwidth]{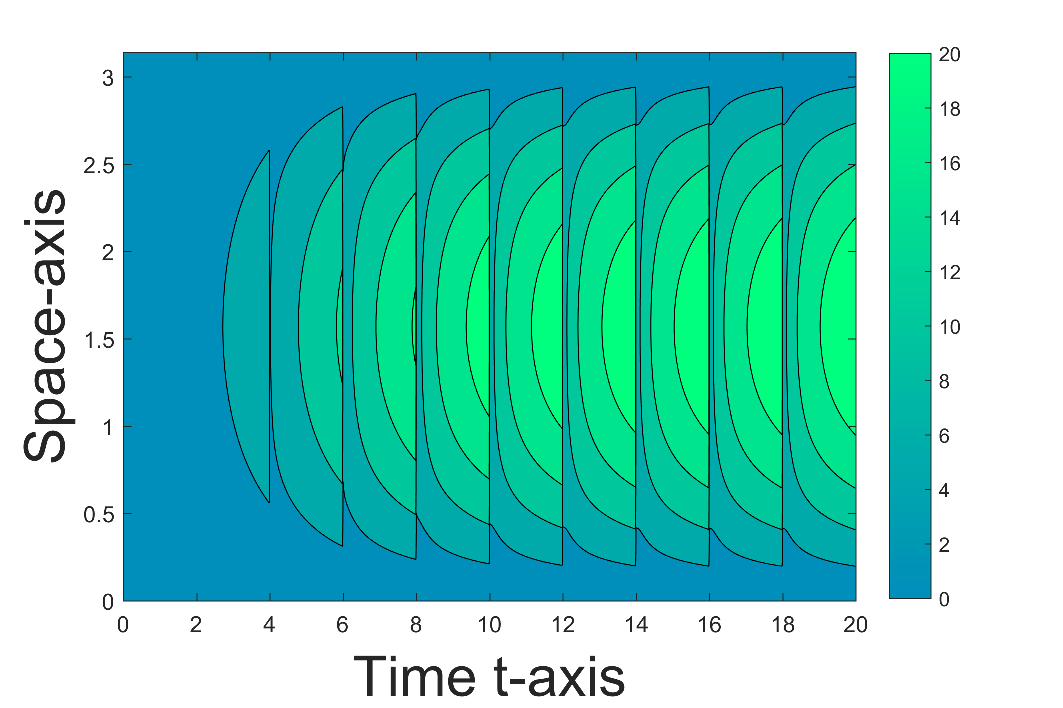}
} }
\subfigure[]{ {
\includegraphics[width=0.4\textwidth]{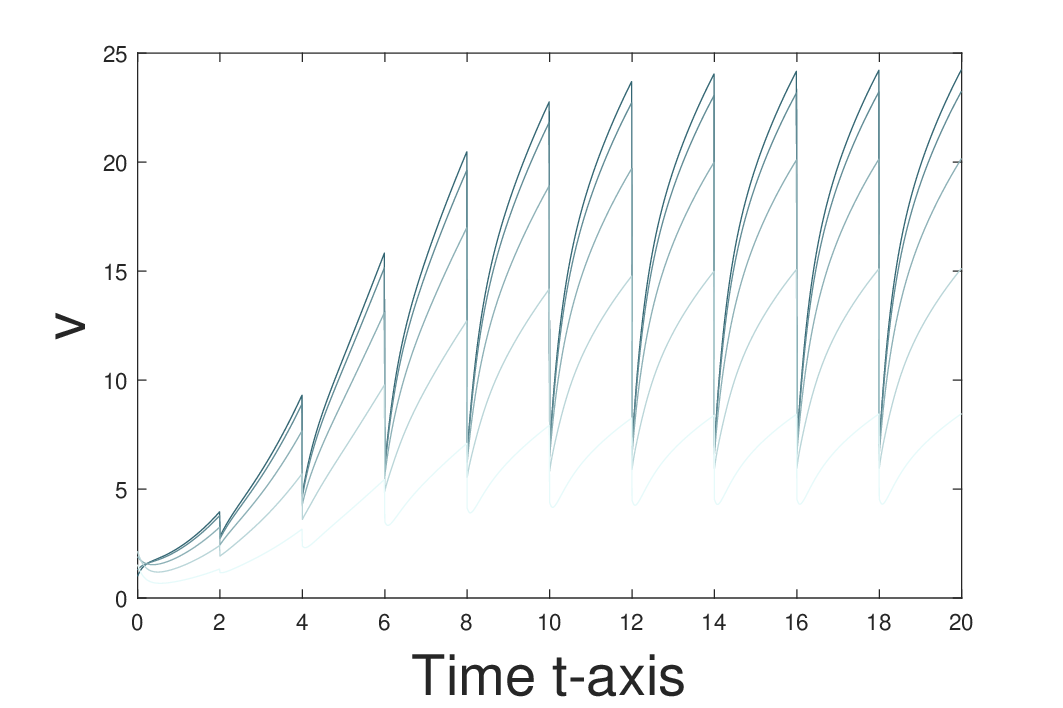}
} }
\renewcommand{\figurename}{\scriptsize Fig.}
\caption{\scriptsize $P(v) = \frac{{9.9v}}{{v + 10}}$ and $\rho (t) \equiv 1$. Graphs (a)-(f) indicate that the variable $v$  ( and $u$ omitted by simplification)  eventually converge to a positive periodic steady state.}
\end{figure}

From the continuity of the principal eigenvalue ${\lambda _1}$ and case \ref{case5.3}, it follows that ${\lambda _1}(0.99) \approx {\lambda _1}(1) \approx  - 0.625 < 0$. Fig.\,6(a)-(f) provide numerical support for Theorem \ref{thm1.3}. While both $u$ and $v$ tend to a positive periodic equilibrium, their densities are substantially lower, as seen by comparing Figs.\,5 and 6. This indicates that impulsive control has a positive effect on containing the infection.

\begin{figure}[htbp]
\centering
\subfigure[]{ {
\includegraphics[width=0.4\textwidth]{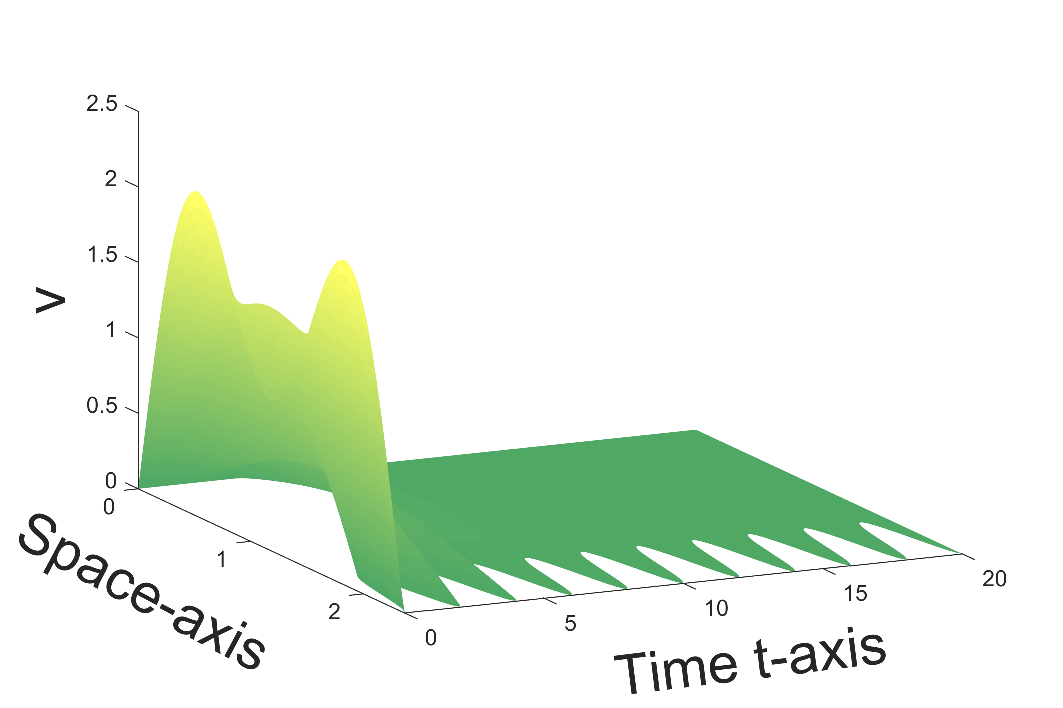}
} }
\subfigure[]{ {
\includegraphics[width=0.4\textwidth]{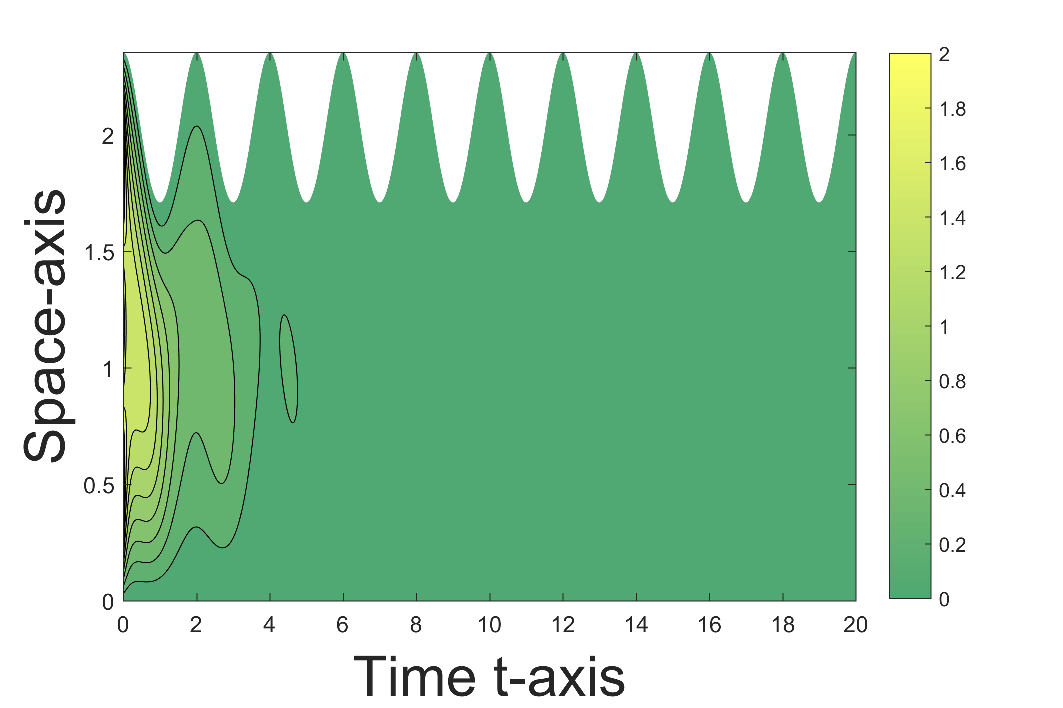}
} }
\subfigure[]{ {
\includegraphics[width=0.4\textwidth]{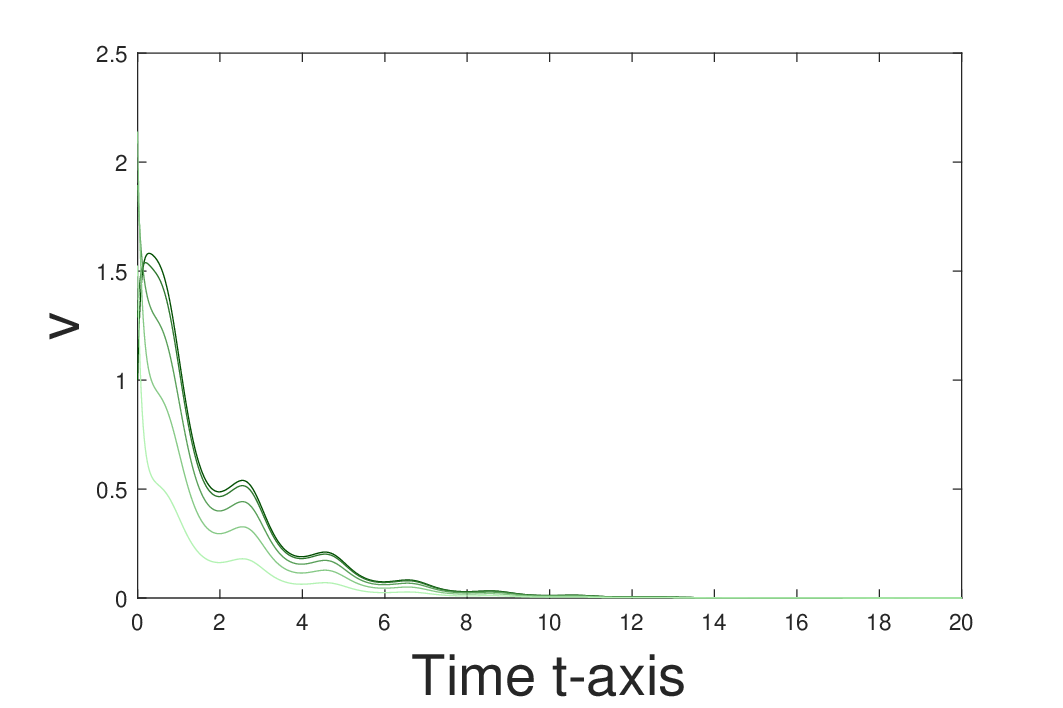}
} }
\renewcommand{\figurename}{\scriptsize Fig.}
\caption{\scriptsize $P(v) = v$ and $\rho (t)={0.75e^{-0.16(1 - \cos \pi t)}}$. Graphs (a)-(f) indicate that the variable $v$  ( and $u$ omitted by simplification) eventually converge to zero.}
\end{figure}
Based on Case \ref{case5.2}(ii), we obtain ${\lambda _1} \ge 0.008 > 0$, which implies that the solution of problem (\ref{a08}) tends to zero. This is supported by the numerical results shown in Fig.\,7(a)-(f). From the contrast between Figs.\,5 and 7, it can be seen that the epidemic, which was initially expanding, disappears after the implementation of a small scaling factor. Consequently, a small scaling factor can rapidly eliminate the spread of dengue fever.

\subsection{Discussion}
As is widely recognized, the periodic variation of climate constitutes an inherent objective law in the natural world. Under such influences, the habitat ranges of mosquitoes and their hosts likewise undergo periodic changes. Impulsive strategies, such as vaccination and regular mosquito suppression, directly affect host and vector populations, thereby altering their abundances or infection states in the short term. To address this, we develop a dengue fever model to investigate the effects of these factors on disease transmission. The initial problem is transformed into a fixed-domain model to facilitate theoretical analysis. However, the introduction of impulses and domain evolution significantly increases the difficulty in analyzing the existence of the principal eigenvalue ${\lambda _1}$ for such periodic eigenvalue problems. Fortunately, we utilize the Poincar$\acute{e}$ map and the Krein-Rutman theorem to resolve this problem. This result is presented in Theorem \ref{thm3.1}.

The theoretical results indicate that if $\lambda _1 <0$, then the solution $(u(y,t),v(y,t))$ to problem (\ref{a08}) converges to the periodic solution $(U(y,t),V(y,t))$ of the corresponding periodic problem (\ref{a09}), which implies that dengue infections will persist in periodic outbreaks. If $\lambda _1 \ge 0$, then $(u(y,t),v(y,t))$ eventually converges to $(0,0)$, that is, the disease will be completely eradicated. It should be noted that previous studies (e.g.,\cite{zq2025,xhy2023}) focus predominantly on the case of $\lambda _1 >0$, while this paper further examines the critical case of $\lambda _1 =0$ in Theorem \ref{thm1.1}.

Numerical simulations presented in Examples \ref{exm5.1} and \ref{exm5.2} illustrate the impacts of impulsive control and domain evolution on dengue transmission dynamics. Specifically, impulsive interventions always positively suppress disease transmission, with their inhibitory effect strengthening as the impulsive intensity $1-P'(0)$ grows. Meanwhile, the scaling factor $\rho (t)$ has a dual role: a relatively small $\overline {{\rho ^{ - 2}}}  =\frac{1}{ \tau }\int_0^\tau  {{\rho ^{ - 2}}(t)} dt$ aggravates disease transmission; conversely, a relatively large $\overline {{\rho ^{ - 2}}}$ helps inhibit epidemic spread. Since the eigenvalue problem incorporates the impulsive term $P'(0)$, however, we cannot derive an explicit expression for the principal eigenvalue $\lambda _1$ and can only obtain estimates under special cases. Consequently, addressing this issue in more complex dynamical models represents an important direction for future research.

\section*{Conflict of interest}
The authors declare there is no conflict of interest.


\begin{thebibliography}{99}
{
\bibitem{Sharif2024}
N. Sharif, S.K. Dey, Infectious Diseases and Global Health Inequity, Springer, New York, 2024.

\bibitem{Dengue2025}
Dengue, https://www.who.int/news-room/fact-sheets/detail/dengue-and-severe-dengue.

\bibitem{Dengue2023-1}
Dengue: The Region of the Americas, https://www.who.int/emergencies/disease-outbreak -news/item/2023-DON475.

\bibitem{Dengue2023}
Dengue: Global situation, https://www.who.int/emergencies/disease-outbreak-news/item/ 2023-DON498.

\bibitem{WHO2026}
World Health Organization: Global Dengue Surveillance Dashboard [Internet], https:// worldhealthorg.shinyapps.io/dengue$\_$global/.

\bibitem{Li2018}
M.Y. Li, An Introduction to Mathematical Modeling of Infectious Diseases, Springer, New York, 2018.

\bibitem{Fischer1970}
D.B. Fischer, S.B. Halstead, Observations related to pathogenesis of dengue hemorrhagic fever. V. Examination of agspecific sequential infection rates using a mathematical model, \textit{J. Biol. Med.}, 1970, 42(5): 329-349.

\bibitem{FengZhilan1997}
Z.L. Feng, J.X. Velasco-Hern$\acute{a}$ndez, Competitive exclusion in a vector-host model for the dengue fever, \textit{J. Math. Biol.}, 1997, 35(5): 523-544.

\bibitem{Coutinho2005}
F.A.B. Coutinho, M.N. Burattini, L.F. Lopez, E. Massad, An approximate threshold condition for non-autonomous system: an application to a vector-borne infection, \textit{Math. Comput. Simulation}, 2005, 70(3): 149-158.

\bibitem{wangwendi2011}
W.D. Wang, X.Q. Zhao, A nonlocal and time-delayed reaction-diffusion model of dengue transmission, \textit{SIAM J. Appl. Math.}, 2011, 71(1): 147-168.

\bibitem{Tewa2009}
J.J. Tewa, J.L. Dimi, S. Bowong, Lyapunov functions for a dengue disease transmission model, \textit{Chaos Solitons Fractals}, 2009, 39(2): 936-941.

\bibitem{zhumin2018}
M. Zhu, Z.G. Lin, Q.Y. Zhang, Coexistence of a cross-diffusive dengue fever model in a heterogeneous environment, \textit{Comput. Math. Appl.}, 2018, 75(3): 1004-1015.

\bibitem{Zitko2014}
T. $\check{Z}$itko, E. Merdi$\acute{c}$, Seasonal and spatial oviposition activity of Aedes albopictus (Diptera: Culicidae) in Adriatic Croatia, \textit{J. Med. Entomol.}, 2014, 51(4): 760-768.

\bibitem{liyuepeng2023}
Y.P. Li, Q. An, Z. Sun, X. Gao, H.B. Wang, Distribution areas and monthly dynamic distribution changes of three Aedes species in China: Aedes aegypti, Aedes albopictus and Aedes vexans, \textit{Parasites $\&$ Vectors}, 2023, 16: 297.

\bibitem{zm2019}
M. Zhu, Y. Xu, J.D. Cao, The asymptotic profile of a dengue fever model on a periodically evolving domain, \textit{Appl. Math. Comput.}, 2019, 362: 124531.

\bibitem{Wangjie2025}
J. Wang, P.Y. Song, S.M. Wang, Y. Zhang, Threshold propagation in a May-Nowak type degenerate reaction-diffusion viral model on periodically evolving domain, \textit{Discrete Contin. Dyn. Syst. Ser. B}, 2025, 30(4): 1415-1440.

\bibitem{Adam2019}
B. Adam, Z.G. Lin, A.K. Tarboush, Asymptotic profile of a mutualistic model on a periodically evolving domain, \textit{Int. J. Biomath.}, 2019, 12(7): 1950078.

\bibitem{Fangjian2024}
J. Fang, Y.F. Li, Y. Su, Population dynamics on periodically evolving domain with periodic growth mechanisms, \textit{SIAM J. Appl. Math.}, 2024, 84(6): 2219-2237.

\bibitem{zhanghan2025}
H. Zhang, M. Zhu, Population dynamics of a logistic model incorporating harvesting pulses on a growing domain, \textit{Commun. Nonlinear Sci. Numer. Simul.}, 2025, 146: 108768.

\bibitem{Wijaya2021}
K.P. Wijaya, J. P$\acute{a}$ez Ch$\acute{a}$vez, T. G$\ddot o$tz, A dengue epidemic model highlighting vertical-sexual transmission and impulsive control strategies, \textit{Appl. Math. Model.}, 2021, 95: 279-296.

\bibitem{Lewis2012}
M.A. Lewis, B.T. Li, Spreading speed, traveling waves, and minimal domain size in impulsive reaction-diffusion models, \textit{Bull. Math. Biol.}, 2012, 74(10): 2383-2402.

\bibitem{Nie2013}
L.F. Nie, Z.D. Teng, B.Z. Guo, A state dependent pulse control strategy for a SIRS epidemic system, \textit{Bull. Math. Biol.}, 2013, 75(10), 1697-1715.

\bibitem{Fazly2017}
M. Fazly, M. Lewis, H. Wang, On impulsive reaction-diffusion models in higher dimensions, \textit{SIAM J. Appl. Math.}, 2017, 77(1): 224-246.

\bibitem{zhangYuron2024}
Y.R. Zhang, T.S. Yi, Y.M Chen, Spreading dynamics of an impulsive reaction-diffusion model with shifting environments, \textit{J. Differential Equations}, 2024, 381: 1-19.

\bibitem{Lu2026}
L. Lu, J.B. Wang, Persistence and spatial propagation of an impulsive integro-differential model with non-local pulse, \textit{J. Math. Biol.}, 2026, 92(1): 14.

\bibitem{Baker2007}
R.E. Baker, P.K. Maini, A mechanism for morphogen-controlled domain growth, \textit{J. Math. Biol.}, 2007, 54: 597-622.

\bibitem{cej1999}
E.J. Crampin, E.A. Gaffney, P.K. Maini, Reaction and diffusion on growing domains: Scenarios for robust pattern formation, \textit{Bull. Math. Biol.}, 1999, 61(6): 1093-1120.

\bibitem{my2021}
Y. Meng, Z.G. Lin, M. Pedersen, Effects of impulsive harvesting and an evolving domain in a diffusive logistic model, \textit{Nonlinearity}, 2021, 34(10): 7005-7029.

\bibitem{wangmingxin2021}
M.X. Wang, Nonlinear Second Order Parabolic Equations, CRC, Boca Raton, 2021.

\bibitem{Ladyzhenskaya1968}
O.A. Ladyzhenskaya, V.A. Solonnikov, N.N. Uraltseva, Linear and quasilinear equations of parabolic type, Academic Press, New York, 1968.

\bibitem{Protter1984}
M.H. Protter, H.F. Weinberger, Maximum Principles in Differential Equations, Springer-Verlag, New York, 1984.

\bibitem{Krein1950}
M.G. Krein, M.A. Rutman, Linear operators leaving invariant a cone in a Banach space, American Mathematical Society, New York, 1950.

\bibitem{pcv2005}
C.V. Pao, Stability and attractivity of periodic solutions of parabolic systems with time delays, \textit{J. Math. Anal. Appl.}, 2005, 304(2): 423-450.

\bibitem{zq2025}
Q. Zhou, Z.G. Lin, C.A. Santos, On an impulsive faecal-oral model in a periodically evolving environment, \textit{Chaos Solitons Fractals}, 2025, 191: 115825.

\bibitem{xhy2023}
H.Y. Xu, Z.G. Lin, C.A. Santos, Spatial dynamics of a juvenile-adult model with impulsive harvesting and evolving domain, \textit{Commun. Nonlinear Sci. Numer. Simul.}, 2023, 122: 107262.

}
\end{thebibliography}
\end{document}